\documentclass[11pt]{amsart}

\usepackage{pb-diagram}

     \usepackage{xcolor}

\usepackage{amsfonts}
\usepackage{amsmath}
\usepackage{amssymb}
\usepackage{verbatim}
\newcommand{\al}{\alpha}
\newcommand{\be}{\beta}
\newcommand{\ga}{\gamma}

\newcommand{\ep}{\epsilon}

\newcommand{\tomega}{\tilde{\omega}}

\newcommand{\dbar}{\overline{\partial}}

\newcommand{\hvarphi}{\hat{\varphi}}

\newcommand{\homega}{\hat{\omega}}

\newcommand{\ddbar}{\sqrt{-1}\partial\dbar}

\newtheorem{theorem}{Theorem}[section]
\newtheorem{proposition}{Proposition}[section]
\newtheorem{lemma}{Lemma}[section]
\newtheorem{example}{Example}[section]

\newtheorem{definition}{Definition}[section]
\newtheorem{corollary}{Corollary}[section]
\newtheorem{remark}{Remark}[section]

\newcommand{\PP}{\mathbb{P}}
\newcommand{\CC}{\mathbb{C}}

\begin{document}

\title{ Connecting toric manifolds by conical  K\"ahler-Einstein metrics}

\author{Ved Datar$^*$, Bin Guo$^{**}$, Jian Song$^\dagger$ and Xiaowei Wang$^{\dagger\dagger}$ }

\address{$*$ Department of Mathematics, Rutgers University, Piscataway, NJ 08854}

\email{veddatar@math.rutgers.edu}

\address{$**$ Department of Mathematics, Rutgers University, Piscataway, NJ 08854}

\email{bguo@math.rutgers.edu}

\address{$\dagger$ Department of Mathematics, Rutgers University, Piscataway, NJ 08854}

\email{jiansong@math.rutgers.edu}

\address{$\dagger\dagger$ Department of Mathematics and Computer Sciences, Rutgers University, Newark, NJ  07102}

\email{xiaowwan@math.rutgers.edu}

\thanks{Research supported in
part by National Science Foundation grants DMS-0847524 and the graduate fellowship of Rutgers University.}

\begin{abstract} We give  criterions for the existence of toric conical K\"ahler-Einstein and K\"ahler-Ricci soliton metrics on any toric manifold in relation to the greatest Ricci and Bakry-Emery-Ricci lower bound. We also show that any two toric manifolds with the same dimension can be joined  by a continuous path of  toric manifolds with conical K\"ahler-Einstein metrics in the Gromov-Hausdorff topology.

\end{abstract}

\maketitle

{\footnotesize  \tableofcontents}

%%%%%%%%%%%%%%%%%%%%%%%%%%%%%%
%
\bigskip

\bigskip
\maketitle

\section{Introduction}

The existence of K\"ahler-Einstein metrics has been a central problem in K\"ahler geometry since Yau's celebrated solution \cite{Y1} to the Calabi conjecture. Constant scalar curvature metrics with conical singularities  have been extensively studied in \cite{Mc, Tr, LT} for Riemann surfaces. In general, one considers a pair $(X, D)$ for an $n$-dimensional compact K\"ahler manifold and a smooth complex hypersurface $D$ of $X$. A conical K\"ahler metric $g$ on $X$ with cone angle $2\pi \beta$ along $D$ is locally equivalent to the following model edge metric
$$g= |z_1|^{-2(1-\beta)} d z_1 \otimes d\bar z_1 +  \sum_{j=2}^{n} dz_j \otimes d\bar z_j  $$
if $D$ is locally defined by $z_1=0$.
Applications of conical K\"ahler metrics are proposed and applied to obtain various Chern number inequalities  \cite{T94, SW}.      Donaldson has developed the linear theory to study the existence of canonical conical K\"ahler metrics  in \cite{D4}. It plays an essential role in the recent breakthrough of the Yau-Tian-Donalson conjecture \cite{T87, CDS1, T13, CDS2, CDS3, CDS4}.   Brendle \cite{Br} solves Yau's Monge-Amp\`ere equations for conical K\"ahler metrics with cone angle $2\pi \beta$ for $\beta\in (0,1/2)$ along a smooth divisor $D$. The general case is settled by Jeffres, Mazzeo and Rubinstein \cite{JMR} for all $\beta\in (0,1)$. As an immediate consequence, there always exist conical K\"ahler-Einstein metrics with negative or zero constant scalar curvature with cone angle $2\pi \beta$ along a smooth divisor $D$ for $\beta\in (0,1)$.
When $X$ is a Fano manifold, Donaldson \cite{D4} proposes to study the conical K\"{a}hler-Einstein equation
\begin{equation}\label{KE-D}
Ric(\omega) = \be \omega + (1-\be) [D],
\end{equation}
where $D$ is smooth simple divisor in the anticanonical class $[-K_X]$ and $\beta \in (0, 1)$.

The solvability of equation (\ref{KE-D})  is closely related to the following holomorphic invariant for Fano manifolds which is known as the greatest Ricci lower bound first introduced by Tian in \cite{T92}.

\begin{definition}\label{Rg}
Let $X$ be a Fano manifold.  The greatest Ricci lower bound $R(X)$ is defined by
\begin{equation}
 R(X)=\sup\{\beta\mid Ric(\omega)\geq \beta \omega, \text{ for some  }\omega\in c_1(X)\cap \mathcal{K}(X)\},%
 \end{equation}
where $\mathcal{K}(X)$ is the space of all K\"ahler metrics on $X$.
\end{definition}
It is proved by Szekelyhidi in \cite{Sze} that $[0, R(X))$ is the maximal interval for the continuity method to solve the K\"ahler-Einstein equation on a Fano manifold $X$. In particular, it is independent of the choice for the initial K\"ahler metric when applying the continuity method. The invariant $R(X)$ is explicitly calculated for $\PP^2$ blown up at one point by Szekelyhidi \cite{Sze}, and for all toric Fano manifolds by Li \cite{LiC}. Recent results \cite{Li13} show  that $R(X)=1$ if and only if $X$ is K semi-stable, and such a Fano manifold satisfies the Chern-Miyaoka inequality \cite{SW}.   It is shown in \cite{SW, LS} that (\ref{KE-D}) cannot be solved for $\beta>R(X)$, answering a question of Donaldson \cite{D4} while it can always be solved for $\beta \in (0, R(X))$ if one replace $D$ by a smooth divisor in the pluri-anticanonical system of $X$. In this paper, we will give various generalizations of the greatest Ricci lower bound.

The Bakry-Emery-Ricci curvature on a Riemannian manifold $(M, g)$ is defined  by
$$Ric_f (g) = Ric(g) + Hess f$$
for a smooth real valued function $f$ on $M$ \cite{BE}. If $(M, g, f)$ satisfies the equaiton $Ric_f(g) = \lambda g$ for some $\lambda \in \mathbb{R}$, it  is called a gradient Ricci soliton with the gradient vector field $V=\nabla f$. We can define the greatest Bakry-Emery-Ricci lower bound on Fano manifolds as an analogue of the greatest Ricci lower bound.

\begin{definition} Let $X$ be a Fano manifold. The greatest Bakry-Emery-Ricci lower bound $R_{BE}(X)$ is defined by
\begin{equation*}
 R_{BE}(X)=\sup\{\beta\mid Ric(\omega)\geq \beta \omega+\mathcal{L}_{Re \xi} \omega, \text{for some }\omega\in c_1(X)\cap \mathcal{K}(X)  \text{ and } \xi\in H^0(X, TX) \}.%
 \end{equation*}
where $\mathcal{L}_\xi$ is the Lie derivative with respect to $\xi$.

\end{definition}

Since $X$ is Fano, it is simply connected and $\mathcal{L}_{Re \xi} \omega =  -\ddbar f_\xi$ for some real-valued smooth function $f_\xi$ with $\nabla_{z_i} \nabla_{z_j} f_\xi =0$ in holomorphic coordinates. This implies that $Ric(\omega) \geq \beta \omega + \mathcal{L}_{Re \xi} \omega$ is equivalent to $$ R_{ij} + \nabla_i\nabla_j f_\xi \geq \beta g_{ij} $$ %
in real coordinates.  Hence
$$R_{BE}(X)= \sup \{ \beta~|~ Ric(\omega) + \ddbar f \geq \beta  \omega,   ~\omega\in c_1(X)\cap \mathcal{K}(X),~  f\in C^\infty(X), \uparrow \dbar f \text{ is holomorphic } \}.   $$
One can relate $R_{BE}(X)$ to the continuity method for solving the K\"ahler-Ricci soliton equation on $X$ as introduced in \cite{TZhu} as analogue of $R(X)$ and explicitly calculate the value of $R_{BE}(X)$ for toric Fano manifolds. In fact, we conjecture that $R_{BE}(X)=1$ for any Fano manifold $X$. However, in this paper, we are more interested in generalizing $R(X)$ and $R_{BE}(X)$ for log Fano manifolds and more specifically, toric conical metrics on toric manifolds.

We start with a few definitions. Let $X$ be an $n$-dimensional toric manifold and  $L$ a K\"ahler class (or equivalently, an ample $\mathbb{R}$  divisor) on $X$.  In \cite{D2, SW}, smooth toric conical K\"ahler metrics are defined and studied in detail and a brief review is given in section 2.  We let $\mathcal{K}_c(X)$ be the set of all smooth toric conical K\"ahler metrics with each cone angle in $(0, 2\pi]$.

\begin{definition}
Let $X$ be a toric manifold. Let $\omega \in \mathcal{K}_c(X)$ be a smooth toric conical K\"ahler metric on $X$.  We say
\begin{equation*}
Ric(\omega) > \al \omega
\end{equation*}
if there exists $\eta\in \mathcal{K}_c(X)$ and an effective toric divisor $D$ such that
\begin{equation*}
 Ric(\omega) = \al\omega + \eta + [D] .
\end{equation*}
\end{definition}
In fact, $\omega$ and $\eta$ have the same cone angles and the divisor $D$ can be explicitly calculated in terms of the cone angles of $\omega$.

\begin{definition}
A smooth toric conical K\"ahler metric $\omega\in \mathcal{K}_c(X)$ is called a  conical K\"ahler-Ricci soliton metric if it satsifies
\begin{equation*}
Ric(\omega) = \al\omega + \mathcal{L}_\xi \omega + [D]
\end{equation*}

\noindent for some holomorphic vector field $\xi$ and effective toric divisor $D$. If $\xi = 0$,  the metric is a   smooth toric conical K\"ahler-Einstein metric.
\end{definition}

 Associated to any toric K\"ahler class, we define the following  geometric invariants $\mathcal{R}(X,L)$, $\mathcal{R}_{BE}(X,L)$  and $\mathcal{S}(X,L)$.

 \begin{definition}  Let $X$ be a toric manifold and $L$ be a K\"ahler class on $X$. Let $\{ D_j \}_{j=1}^N$ be the set of all prime toric divisors on $X$. Then we define
 \medskip

 \begin{enumerate}

 \item $\mathcal{R}(X,L) = \sup \{\al ~ |~ Ric(\omega) > \al\omega  \text{ for some } \omega \in c_1(L)\cap\mathcal{K}_c(X)   \},$
 \medskip
\item $\mathcal{R}_{BE}(X,L) = \sup \{\al ~ |~   Ric(\omega)+ \mathcal{L}_{\xi} \omega> \al\omega \text{ for a } \omega\in c_1(L)\cap\mathcal{K}_c(X) \text{ and a toric } \xi \in H^0(X, TX)  \},$
\medskip

\item $\mathcal{S}(X,L) = \sup{\{\al ~|~ \text{there exists } D = \sum_{j=1}^{N}{a_jD_j} \sim -K_X - \al L  \hspace{0.05in}\text{with} \hspace{0.05in}a_j \in [0,1)\}}$ .

\end{enumerate}

\end{definition}

$\mathcal{R}(X,L)$ and $\mathcal{R}_{BE}(X,L)$ are  natural generalizations of $R(X)$ and $R_{BE}(X)$ for log Fano manifolds with polarization $L$. $\mathcal{S}(X, L)$ characterizes when $(X, D)$ is log Fano as by definition $K_X +D$ is klt and negative. In the special case that $X$ is toric  Fano and $L = -K_X$, $\mathcal{R}(X,-K_X)$ is the usual greatest Ricci lower bound studied in \cite{Sze} and $\mathcal{S}(X,-K_X) = 1$.  In fact, for any toric pair $(X, L)$, $\mathcal{R}(X, L)$ and $\mathcal{S}(X,L)$ are both positive.  In general, one can define $\mathcal{R}(X,L)$ and $\mathcal{R}_{BE}(X,L)$ for any log Fano pair $(X,L)$ by requiring $Ric(\omega) - \alpha \omega\geq 0$ and $Ric(\omega)+\mathcal{L}_\xi \omega - \alpha \omega\geq 0$ in the current sense.

Any toric manifold $X$ is induced by an integral Delzant polytope $P$ and $P$ determines a K\"ahler class on $X$. Without loss of generality, we let
\begin{equation}
P = \{ x\in \mathbb{R}^n~|~ l_j(x) >0, ~j=1, ..., N\},
\end{equation}
where $l_j(x)=v_j \cdot x + \lambda_j$, $v_j$ is a prime integral integral vector in $\mathbb{Z}^n$ and $\lambda_j \in \mathbb{R}$ for all $j=1, ..., N$.  As a special case, when $X$ is Fano, one can choose $\lambda_j=1$ for all $j$ and the polytope gives the anti-canonical polarization of $X$. The existence of smooth toric K\"ahler-Einstein and K\"ahler-Ricci soliton metrics on toric Fano manifolds is completely settled by Wang-Zhu \cite{WZ}. We generalize their results to toric conical K\"ahler-Einstein and K\"ahler-Ricci soliton metrics on any toric manifold.

\begin{theorem} Let $X$ be an $n$-dimensional toric K\"ahler manifold and $L$ be the  K\"ahler class on  $X$ induced by the Delzant polytope $P$.   Then

\begin{enumerate}

\item   $\mathcal{R}_{BE}(X,L) =\mathcal{S}(X,L)>0$ and
\begin{equation}\mathcal{R}_{BE}(X,L)  = \sup{\{\al ~|~ \text{there exists } \tau \in P \hspace{0.05in}\text{with}\hspace{0.05in} 1 - \alpha l_j(\tau) > 0, j=1, ..., N\}},
\end{equation}

\item For any $\al \in (0,  \mathcal{S}(X,L) )$ and $\tau \in P$ satisfying $1 - \al l_j(\tau) \geq 0$ for all $j$, there exists a unique $\omega \in L \cap \mathcal{K}_c(X)$ solving the K\"ahler-Ricci soliton equation
\begin{equation}\label{kreqn}Ric(\omega) = \alpha\omega + \mathcal{L}_{\xi} \omega + [D].
\end{equation}
Moreover the divisor $D$ and the vector field $\xi$ are given by
\begin{equation}
D = \sum_{j=1}^{N}{(1-\al l_j(\tau))D_j},~~
\xi = \sum_{i=1}^{n}{c_iz_i\frac{\partial}{\partial z_i}},
\end{equation}
\noindent where $z_i$'s are the standard coordinates on $(\mathbb{C}^{*})^n$ and $c\in \mathbb{R}^n$ is uniquely given by
\begin{equation}
\tau = \frac{\int_{P}{xe^{c\cdot x} \,dx}}{\int_{P}{e^{c\cdot x} \,dx}}.
\end{equation}
\item There does not exist a toric conical K\"ahler-Ricci soliton metric $\omega\in L\cap \mathcal{K}_c(X)$ solving the soliton equation (\ref{kreqn}) for any $\alpha> \mathcal{R}_{BE}(X,L)$.

\end{enumerate}
\label{theorem-1}
\end{theorem}

The existence of  toric conical K\"ahler-Ricci soliton metrics on log Fano toric varieties is derived in \cite{BB} and for general toric manifolds by allowing the cone angle in $(0, \infty)$ \cite{Leg}. Our result gives a complete classification for  the existence of toric conical K\"ahler-Einstein and K\"ahler-Ricci soliton metrics using the invariants $\mathcal{R}(X,L)$ and $\mathcal{R}_{BE}(X,L)$ for any K\"ahler class. We are only interested in the toric conical K\"ahler metrics with cone angle in $(0, 2\pi)$ since the smooth part is geodesic convex and various Riemannian geometric properties can be applied. In particular, it gives  optimal regularity and a complete classification of smooth toric conical K\"ahler-Ricci soliton metrics for any toric pair $(X, L)$. The proof is based on a family version of the $C^0$ estimate and we will apply it to toric degenerations for smooth toric conical K\"ahler-Einstein metrics.

As a special case, we obtain an existence result for conical K\"ahler-Einstein metrics on toric manifolds and apply it to characterize the invariant $\mathcal{R}(X, L)$ in terms of the polytope data.

\begin{theorem} Let $X$ be an $n$-dimensional toric K\"ahler manifold and $L$ be the  K\"ahler class on  $X$ induced by the Delzant polytope $P$.      Let $P_C$ be the barycenter of $P$. Then
\begin{enumerate}

\item  $\mathcal{R}(X,L)>0$  and
\begin{equation}
\mathcal{R}(X,L) = \sup{\{ \al ~|~ 1-\al l_j(P_C) > 0, ~j=1, ..., N\}}.
\end{equation}

\item For all $\al \in (0, \mathcal{R}(X, L)] $, there exists a unique toric conical K\"ahler-Einstein metric $\omega\in L \cap \mathcal{K}_c(X)$ solving
\begin{equation}\label{keeqn}
Ric(\omega)=\alpha\omega + [D].
\end{equation} Moreover the divisor $D$ is  given by
\begin{equation}
D = \sum_{j=1}^{N}{(1-\al l_j(P_c))D_j} .
\end{equation}

\item  There does not exist a toric conical K\"ahler-Einstein metric $\omega\in L\cap \mathcal{K}_c(X)$ solving the equation (\ref{keeqn}) for any $\alpha> \mathcal{R}(X,L)$.

\end{enumerate}
\label{theorem-2}
\end{theorem}

In the special case when $X$ is Fano and  $L= -K_X$, $l_j(0) = 1$ and so $1-\alpha l_j(P_C) = (1-\alpha)l_j(\frac{-\alpha P_C}{1-\alpha})$ . By the theorem, $\mathcal{R}(X,L)$ is the maximum of all $\alpha$ such that  $\frac{-\alpha P_C}{1-\alpha}$ remains inside the polytope, generalizing  the results in the smooth case in \cite{Sze} and \cite{LiC}. There is a subtle difference between Theorem \ref{theorem-1} and Theorem \ref{theorem-2} that in Theorem \ref{theorem-2}, $\alpha$ can be taken to be $\mathcal{R}(X,L)$ while in Theorem \ref{theorem-1}, $\alpha$ has to be strictly less than $\mathcal{R}_{BE}(X,L)$. Such a phenomena will be explained in  Example \ref{ex5}.

Smooth K\"ahler-Einstein and K\"ahler-Ricci soliton metrics on Fano manifolds are unique, while the space of conical K\"ahler-Einstein and K\"ahler-Ricci soliton metrics is much bigger. One would ask under what assumptions the space of conical K\"ahler-Einstein metrics is connected. In other words, given two log Fano manifolds $X_0$ and $X_1$, we ask when and how one can connect $X_0$ and $X_1$ by a family of conical K\"ahler-Einstein spaces in Gromov-Hausdorff topology. Now we state our main theorem of the paper to answer such a question in the toric case.

\begin{theorem}Let $X_0$ and $X_1$ be two $n$-dimensional toric manifolds. Suppose $\omega_0\in \mathcal{K}_c(X_0)$ and $\omega_1 \in \mathcal{K}_c(X_1)$ are two smooth toric conical K\"ahler-Einstein metrics on $X_0$ and $X_1$ respectively. Then, there exist a family $\{ ( X_t, \omega_t )\}_{t\in[0,1]}$ of $n$-dimensional toric manifolds $X_t$ with smooth toric conical K\"ahler-Einstein metrics $\omega_t \in \mathcal{K}_c(X_t)$ for  $t\in[0,1]$, such that

\begin{enumerate}

\item $(X_t,\omega_t)$ is a continuous path in  Gromov-Hausdorff topology for $t\in [0,1]$,

\medskip

\item $\omega_t$ is piecewise smooth in $t$ on the complex torus $(\mathbb{C}^*)^n$.

\end{enumerate}

\label{theorem-3}

\end{theorem}

Theorem \ref{theorem-3} can be considered to be an analytic analogue of the weak factorization theorem for toric varieties in algebraic geometry. Combined with Theorem \ref{theorem-2}, it implies that any two toric manifolds of same dimension can be joined by a continuous path of conical K\"ahler-Einstein spaces in Gromov-Hausdorff topology.  It is a natural question to ask if for any two birationally equivalent Fano manifolds, there exists a continuous path connecting them by  Fano varieties coupled with conical K\"ahler-Einstein metrics, in Gromov-Hausdorff topology. This is related to the connectedness of moduli space of log Fano varieties coupled with conical K\"ahler-Einstein metrics.

\section{Preliminaries}

\subsection{Basics of toric varieties }
\def\ZZ{\mathbb{Z}}
\def\RR{\mathbb{R}}
\def\CC{\mathbb{C}}
\def\PP{\mathbb{P}}
\def\QQ{\mathbb{Q}}
\def\ep{\epsilon}
\def\sO{\mathcal{O}}

\def\Span{\mathrm{span}}
\def\Rm{\mathrm{Rm}}
\def\Hom{\mathrm{Hom}}

In this section, let us recollect some well-known facts of projective toric varieties.

\begin{definition}\label{Dez}
A convex polytope $P\subset\RR^n$  is called a {\em Delzant polytope} if a neighborhood of any vertex  $p\in P$ is  $SL(n,\ZZ)$ equivalent to $\{x_j\geq 0, j=1, ..., n\} \subset \RR^n $. $P$ is called {\em an integral  Delzant} polytope if each vertex $p\in P$ is a lattice point in $\ZZ^n\subset \RR^n$.
\end{definition}

Let $P$ be an integral Delzant polytope in $\mathbb{R}^n$ defined by
\begin{equation}
P=\{ x\in \mathbb{R}^n~|~ l_j (x) >0, j=1, ..., N\},
\end{equation}
where $$l_j(x) = v_j \cdot x +  \lambda_j$$
and  $v_i $ is a primitive integral vector in $\mathbb{Z}^n$ and $\lambda_j \in \ZZ_+$ for all $j=1, ..., N$.  Let  $\Sigma
_{P}$ be the fan consisting of the cones over the faces of the {\em polar polytope}
$$
\check P=\{y\in \RR^n\mid \langle y,x\rangle_\RR\geq -1\text{ for all } x\in P\} .
$$
Then $\Sigma_P$ defines an $n$-dimensional smooth projective toric variety $X_P$. Its Picard group $\mathrm{Pic}(X)$ is generated by $D_i$'s, the toric divisors corresponding to the generators of edges $e_i$'s of the fan $\Sigma_P$.  For any toric divisor $D=\sum a_{i}D_{i}$, it determines a rational
convex polyhedron
\begin{eqnarray*}
P_{D}
&=&\{ \al\in \RR^n \mid \langle \al,e_{i}\rangle \geq-a_{i}\text{ for all }i\}\subset \RR^n\ ,
\end{eqnarray*}
and the space of global sections of the line bundle $\sO_X(D)$ is given by
\begin{equation}
H^{0}( X,\sO_X( D) ) =\bigoplus\limits_{\al\in
P_{D}\cap \ZZ^n}\mathbb{C\cdot \chi }^{\al},
\end{equation}%
where $\chi$'s are the characters $\Hom(T,\CC^*)$. In particular, we have
\begin{equation}\label{RR}
\dim H^{0}( X,\mathcal{O}( kD) ) =k^n\mathrm{Vol}(P_D)+O(k^{n-1}),
\end{equation}
where $\mathrm{Vol}(P_D)$ denote the Euclidean volume of $P_D\in \RR^n$.
In particular,  for the anti-canonical divisor  $-K_{X_P}=\sum_{i}D_{i}$, we have
$
P_{-K_X}=\{ \al\in \RR^n \mid \langle \al,e_i\rangle \geq -1\text{ for all }i\},
$
hence $\mathrm{Vol}(P_{-K_X})>0$. By \eqref{RR}, we conclude with the following well-known lemma.
\begin{lemma} \label{big}
Let $X$ be a smooth projective toric variety, then $-K_X$ is  big.
\end{lemma}

Next we study how far away is a toric variety $X$ from being Fano. To do that, we need to introduce some notions.
\begin{definition}\label{log-fano}
Let $X$ be a smooth projective variety and $D=\sum_i a_iD_i\geq 0$ be a $\QQ$-divisor with $D_i$'s being prime divisors, so that $a_i$'s are non-negative rational numbers.
If $(X,D)$ is a log pair, then we say that a birational map: $\pi:(Y,D')\to (X,D)$ is a log resolution if the strict transform of  $D$ union the exceptional locus is {\em log smooth}, i.e. normal crossings of smooth divisors. We say that $(X,D)$  is
{\em Kawamata log terminal (klt)} if when we write
$$ K_Y+D'= \pi^\ast(K_X+D)$$
then $\lfloor D'\rfloor \leq 0$, that is, if we write $D'=\sum a'_i D'_i$ with $D'_i$'s being prime divisors  then $a'_i<1$ for all $i$.
\end{definition}

Here we list the properties of being klt.
\begin{lemma}\label{klt} \cite[Lemma 4.2]{Mck}
 Let $X$ be  a smooth projective variety
 \begin{enumerate}
 \item If $D\geq 0$ is an effective $\QQ$-divisor  then there is a $\ep>0$ such that $(X,\ep D)$ is klt.
 \item\label{ep-D} If $(X,D)$ is klt and $D+D'\geq 0$  then $(X,D+\ep D')$ is klt for any small $\ep>0$.
 \item\label{num} If $(X,D)$ is klt and $G$ is semiample then we may find $G'\equiv G$ such that $K_X+D+G'$ is klt.
 \end{enumerate}
\end{lemma}

%\begin{proof} Since $X$ is smooth already, for any log resolution $\pi:Y\to X$, we have
%$$K_Y=\pi^\ast K_X+E$$
%with $E\geq 0$, hence $(X,\emptyset)$ is clearly klt.  On the other hand, we have
%$$K_Y+D'_t=\pi^\ast (K_X+tD)$$
%by continuity, for small $t$, we have $\lfloor D_t'\rfloor \leq 0$.
%\end{proof}
\begin{remark}\label{toric} One notices that, if $X$ is toric then all the divisors in the statements above can be  chosen to be toric divisors.
\end{remark}

Now we are ready to introduce the main terminology of this section

\begin{definition}Let $X$ be a smooth projective variety.We say that $X$ is a {\em log Fano variety } if there is a divisor $D$ such
that $K_X+D$ is  {\em klt} and $-(K_X+D)$ is ample.
\end{definition}

In order to characterize the log Fano varieties, we have the following  criterion.

\begin{lemma}\label{cri}
Let $X$ be a smooth projective variety, then $X$ is log Fano if and only if there is an effective big divisor $D$ such that $K_X+D$ is klt and numerically trivial.
\end{lemma}
\begin{proof}
First, we prove the sufficiency. Suppose $(X,G)$ is log Fano then $K_X+G$ is klt and $-(K_X+G)$ is ample.   So there is an ample divisor $A$ such that $-(K_X+G)=A/m$ for some $m\in \ZZ$.  Clearly, we have $K_X+G+A/m\equiv 0$.   Since $A/m$ is ample, by Lemma \ref{klt} there is a  $A'\equiv A$ such that  $K_X+G+A'/m$ is klt. Now we define  $D:=G+A'/m\geq 0$, which is big (cf. \cite[Lemma 2.60]{KM98}) then $K_X+D\equiv 0$ and it is klt.  This finishes the proof of sufficiency.

Conversely, by our assumption, $D$ is big, so $D\sim A+B$ with $A$ being ample and $B\geq 0$ (cf. \cite[Lemma 2.60]{KM98}). Now  we define
$$K_X+G:=K_X+(1-\epsilon) D+\epsilon B=K_X+D+\epsilon(B-D)\ .$$
Let  $D'=B-D$ then $D+D'=B\geq 0$, by part \ref{ep-D} of Lemma \ref{klt} we conclude $K_X+G=K_X+D+\ep D'$ is klt for sufficient small $\ep$.
On the other hand, since $K_X+D\equiv 0$, $-(K_X+G)\equiv -\ep D'\sim\ep A$ is ample, hence $(X, G)$ is log Fano  by Definition \ref{log-fano}. Our proof is completed.

\end{proof}

As a direct consequence, when  $X$ is toric,  $-K_X$ is effective and big by Lemma \ref{big}.  Let $D=-K_X$ and apply Lemma \ref{cri} above, we obtain a toric (cf. Remark \ref{toric} ) $G\geq 0$ such that $K_X+G$ is log canonical and ample. Since $D$ is big, we may choose ample divisor $A$ and $\epsilon$ such that  $D-\ep A$ is still big and $\lfloor D-\ep A\rfloor\leq 0$. This implies  $(X,D-\ep A)$ is klt. Now $\ep A$ is ample, by part \ref{num} Lemma \ref{klt} there is a $A'\equiv A$ such that for $D'=D-\ep A+\ep A'$, $(X,D')$ is klt. Now $D'\equiv D$  is also big and $K_X+D'\equiv 0$. By Lemma \ref{cri}, we obtain $X$ is log Fano, which is well-known (c.f. \cite{Mck}). In the following, we give an elementary  proof without using Lemma \ref{cri}.

\begin{theorem} Let  $X$ be a projective toric variety and $D\subset X$ be any ample divisor,  there is a toric divisor  $G\sim mD$ for some $m\in \ZZ$ such that $(X, -K_X-\ep G)$ is log Fano for all $0<\ep \ll 1$.
\end{theorem}
\begin{proof}
Since $D$ is ample, for $m$ large, there is a  $0\leq G'\in |mD|$, the linear system of $mD$. Take a general one parameter subgroup $\lambda\subset (\CC^\ast)^n$ and let $G=\lim_{t\to 0}\lambda(t)\cdot G'$, then $G\geq 0$ is $G$ is toric.  So $G=\Sigma_{i=1}^N a_i D_i$ with all $a_i\geq0 $, where $D_i$'s  are the toric divisors generating  $\mathrm{Pic}(X)$.

We claim that $-K_X-\ep G\geq 0$ and $\lfloor -K_X-\ep G\rfloor \leq 0$ for $0<\ep\ll1$, from which we obtain $(X,  -K_X-\ep G)$ being log Fano. But this follows from the fact $-K_X-\ep G=\sum_{i=1}^N (1-\ep a_i) D_i>0$ and $\lfloor(1-\ep a_i)\rfloor =0$ for any  $0<\ep\ll1$. Hence the proof is completed.
\end{proof}

The following weak factorization theorem first proved in \cite{Wl03} (cf.  also \cite{AKMW02}) reduces the proof of our main result to the case of a simple blow-up or blow-down of a smooth toric center.
\begin{theorem} Let $f:X\dashrightarrow Y$ be a toric birational map between two complete nonsingular  toric varieties $X$ and $Y$ over $\CC$, and let $U\subset X$ be an open set where $f$ is an isomorphism. Then $f$ can be factored into a sequence of blow-ups and blow-downs with nonsingular irreducible toric centers disjoint from $U$,  namely, there is a sequence of birational maps between complete nonsingular toric varities
$$ X=X_0 \stackrel{f_1}{\dashrightarrow }X_1\stackrel{f_2}{\dashrightarrow }\cdots\stackrel{f_i}{\dashrightarrow }X_i\stackrel{f_{i+1}}{\dashrightarrow }\cdots \stackrel{f_n}{\dashrightarrow }X_n=Y, $$
where
\begin{enumerate}
\item $f=f_n\circ f_{n-1}\circ \cdots \circ f_2\circ f_1$,
\item $f_i$ is an isomorphism on $U$, and
\item either $f_i: X_{i-1}\dashrightarrow X_i $ or $f_i^{-1}:X_{i}\dashrightarrow X_{i-1}$ is a morphism obtained by blowing up a nonsingular irreducible toric center disjoint from $U$.
\end{enumerate}
\label{factorization theorem}
\end{theorem}

\subsection{Toric conical K\"ahler metrics}\label{sec:toric}

\medskip
In this subsection, we introduce the conical K\"ahler metrics on a smooth projective toric variety $X$ following \cite{SW}. On the open dense orbit $X\supset (\mathbb{C}^*)^n \cong \mathbb{R}^n \times (S^1)^n$, it is convenient to use logarithmic and  angular coordinates $(\rho_i=\log{|z_i|^2},\theta_i)$ for $(z_1, \dotsc ,z_n) \in (\mathbb{C}^*)^n $. By the $\partial\bar\partial$-Lemma, for any K\"ahler metric $\omega$ on $X$, there is a $\varphi\in C^\infty((\CC^*)^n)$ such that $\omega|_{(\CC^*)^n}= \ddbar\varphi$.  If we assume further that $\omega$ is $(S^1)^n$-invariant, then $\varphi = \varphi(\rho_1,\dotsc ,\rho_n)$ is strictly convex, and  we have  formulas
\begin{align*}
\omega &= \frac{1}{4}\frac{\partial^2 \varphi}{\partial \rho_i\partial \rho_j} d\rho_i \wedge d\theta_j \\
Ric(\omega) &= -\frac{1}{4}\frac{\partial^2 \log(\det (\nabla^2 \varphi)}{\partial \rho_i \partial \rho_j}d\rho_i \wedge d\theta_j,
\end{align*}
 where as usual we sum over repeated indices.  Moreover, the moment map for the Hamiltonian action of $(S^1)^n$ on $X$ is up to a constant precisely given by  $\nabla \varphi:X\to \RR^n$, whose image is a bounded convex polytope  $P\subset \RR^n$ by the famous Atiyah and Guillemin-Sternberg convexity theorem.  Conversely,  a theorem of Delzant says that if a polytope $P$ is integral Delzant (cf. Definition \ref{Dez}), one can always associate a toric variety $X_P$ (,as we introduced in the beginning of this section) and a polarization $L \rightarrow X_P$ with a K\"ahler metric $\omega \in c_1(L)$ such that $\nabla \varphi(\mathbb{R}^n) = P$.

Let $\varphi_P = \log (\sum_{\alpha \in \mathbb{Z}^n\cap \overline{P}} |z^\alpha |^{2} )$. Then $\omega_P = \ddbar \varphi_P$ is a smooth K\"ahler metric on $(\mathbb{C}^*)^n$ and it can be smoothly extended to a smooth global toric K\"ahler metric on $X_P$ in $c_1(L)$. Then the space of toric K\"ahler metrics in $c_1(L)$ is equivalent to the set of all smooth plurisubharmonic function $\varphi$ on $(\mathbb{C}^*)^n$ such that $ \varphi- \varphi_P$ is bounded and $\ddbar \varphi$ extends to a smooth K\"ahler metric on $X_P$.  Let $u$ be the Legendre transform of $\varphi$. The recall that $u$ and $\varphi$ are related by
$$\varphi (\rho) = \mathcal{L} u (\rho) = \sup_{x\in P} (x\cdot \rho - u(x)), ~ u (x) = \mathcal{L} \varphi(x) = \sup_{\rho\in \mathbb{R}^n } (x\cdot \rho - \varphi(\rho))$$
or equivalently
$$\varphi(\rho)= x\cdot \rho - u(x), ~~~ u(x) = x\cdot \rho - \varphi(\rho), ~~~x=\nabla_\rho \varphi(\rho), ~\rho= \nabla_x u(x).$$

The \textit{Guillemin boundary conditions} imply that $ \ddbar \varphi$ extends to a global K\"ahler metric on $X_P$ if and only if
\begin{equation}
u =  u(x) = \sum_{j=1}^N l_j (x) \log l_j(x) + f(x)
\end{equation}

\noindent for some $f\in C^\infty(\overline P)$ and $u$ is strictly convex on $P$.

Now we extend the Guillemin condition to toric conical K\"ahler metrics on $X_P$,  a generalization of orbifold K\"ahler metrics. On each coordinate chart determined by the pair $(p, \{v_{p, i}\}_{i=1}^n)$  associated to a vertex of $P$, we let $z=(z_1, ..., z_n)$ be the coordinates on $\CC^n$. The closure of $\{ z_i =0\}\subset X$ give rise to a smooth toric divisor of $X_P$. Let $D$ be a toric divisor of $X_P$ and suppose $D$ restricted to this coordinate chart is given by $$\sum_{i=1}^n a_i \{z_i=0\}. $$
For any function $f(z)$  invariant under the $(S^1)^n$-action, we can  lift it to a function
$$ \tilde f(w) = f(z)$$
by letting
$$|w_i| = |z_i| ^{\beta_i}, ~w=(w_1, ..., w_n) \in \mathbb{C}^n, $$
and clearly $\tilde f(w)$ is also $ (S^1)^n$-invariant.  $w\in \mathbb{C}^n$ can be regarded a $\beta$-covering of $z\in \mathbb{C}^n$. Now we introduce  consider the $(S^1)^n$-invariant function space for $k\in \mathbb{Z}_{+}$ and $\gamma \in [0,1]$
$$C^{k, \gamma}_{\beta, p} = \{ f(z)=f(|z_1|, ..., |z_n|) ~|~ \tilde f(w) \in C^{k, \gamma} (\mathbb{C}^n)  \} .$$
This in turn defines the weighted function space
$$C^{k, \gamma}_{\beta}(X_P), \beta=(\beta_1, ..., \beta_N)\in (\mathbb{R}_+)^N $$ whose restriction on each chart belongs to $C^{k, \gamma}_\beta$ with respect to the weight $\beta$ and $\beta_j$ corresponding to the divisor induced by $l_j(x)=0$.
\begin{definition}
A K\"ahler current $\omega \in \text{c}_1(L)$ is said to be a \textit{smooth $\be$-conical} metric if for each vertex $p$ of the polytope $P$,
\begin{equation*}
\omega|_{U_p} = \ddbar\varphi_p
\end{equation*}

\noindent for some $\varphi_p \in C^{\infty}_{\beta, p}$. Such a metric naturally has a cone angle of $2\pi\be_j$ along the divisor $D_j$.
\label{smooth conical metric}
\end{definition}

\noindent The local lifting $\tilde \varphi (w)$ is a smooth plurisubharmonic function on the lifting space $w\in \mathbb{C}^n$. We can also define the space of weighted toric K\"ahler potential $\varphi$ on $(\mathbb{C}^*)^n$ such that $\varphi - \varphi_P$ is bounded and $\ddbar \varphi$ extends to a smooth weighted K\"ahler metric on $X_P$.  In particular, we have the following conical extension of the Guillemin condition for toric K\"ahler metrics.

\begin{proposition}
Let $\varphi$ be a toric potential on $(\mathbb{C}^*)^n$, $u$ it's Legendre transform and $P$ the image of $\mathbb{R}^n$ under $\nabla\varphi$. Then $\omega = \ddbar\varphi$ extends to a global smooth $\be$-conical metric on $X_P$ if and only if
\begin{equation}
u(x) = \sum_{j=1}^N\beta_j^{-1} l_j(x) \log l_j(x) + f(x)
\label{symplectic potential}
\end{equation}
 for some $f \in C^{\infty}(\overline P)$ and $0 \leq \be_j \leq 1$. Moreover, the angle along the divisor corresponding to $l_j$ is precisely $2\pi\be_j$.

\end{proposition}

The main advantage of dealing with conical metrics on toric manifolds is that one has all the curvature bounds.

\begin{lemma}{\em \cite{SW}}
Let $g$ be a smooth  toric conical metric on a toric manifold $X$. Let $D$ be a toric divisor consist of all toric prime divisors and let $\Rm$ denote the full curvature tensor of $g$. Then for any $k\geq 0$ there exists a constant $C_k$ such that for all $p\in X\backslash D$,
\begin{equation}
|\nabla^k_g \Rm|_{g}^2(p) \leq C_k.
\end{equation}
\label{curvature bound}
\end{lemma}

\subsection{Comparison theorems for conical K\"ahler manifolds}

%\section{Comparison theorems for conical K\"ahler manifolds}
In this subsection, we will state comparison theorems for cone metric with Ricci curvature bounded below, which will be needed to prove the Gromov-Haursdorff convergence in section 4.
\subsubsection{Geodesic convexity}
Assume $D$ is a smooth divisor, $g$ is a cone metric with cone angle $2\pi \alpha$ $(\alpha\in (0,1))$ along $D$. Around a point $p\in D$, we choose local coordinates $\{z_1,z_2,\ldots, z_n\}$ with $D=\{z_1=0\}$ and $p=(0,\ldots, 0)$. Set $w_1= z_1^\alpha/\alpha$, which can be viewed a singular holomorphic change of coordinates. Although $w_1$ is multi-valued, we can work locally in the logarithmic Riemann surface which uniformizes this variable. Written in terms of the coordinates $\{w_1,z_2,\ldots,z_n\}$, the model flat metric $g_\alpha$ can be written as
\begin{equation*}
g_\alpha = |d w_1|^2 + \sum_{k=2}^n |dz_k|^2,
\end{equation*}
which is the standard flat metric on $\mathbb C^n$.

% extended smoothly to the origin $\tilde{p}$, where $\tilde p$ corresponds to $p$ in the original coordinates.

It is well-known that $(X\backslash D,g)$ is geodesic convex, i.e., any minimal geodesic $\gamma$ connecting two points $p_1, p_2\in X\backslash D$ is strictly contained in $X\backslash D$. For reader's convenience, we present a proof of the geodesic convexity.
\begin{lemma}\label{lemma:geod}
Let $(X, D)$ be given as above, $g$ a smooth cone metric with cone angle $2\pi \alpha$ along $D$ with $\alpha\in (0,1)$. Let $\gamma : [0,1]\to X$ be a minimal geodesic with $\gamma (0) =p\in X\backslash D $, $\gamma (1)=q\in X\backslash D$. Then $\gamma|_{[0,1]}\subset X\backslash D$.
\end{lemma}
\begin{proof}
We only need to show that if $\gamma: (-\delta,\delta)\to X$ is a minimal geodesic with $\gamma(0)=s\in D$, then there exist an $\varepsilon>0$ such that $\gamma|_{(-\varepsilon, \varepsilon)}\subset D$, from which it follows that $\gamma|_{(-\delta,\delta)}\subset D$ by the closedness of $D$. Recall that in appropriate coordinates $\{z_1,\cdots, z_n\}$ centered at a point in $D$, the conic metric $g$ can be approximated by the model flat metric $g_\alpha$ in the sense that there exists a constant $\varepsilon(r)>0$ with $\varepsilon(r)\to 0$ as $r\to 0$ such that
\begin{equation*}
(1-\varepsilon(r))g_\alpha(z)\le g(z)\le (1+\varepsilon(r)) g_\alpha(z), \quad \text{on } \{|z|\le r, z_1\neq 0\}.
\end{equation*}
For two points $p,q$ outside $D$ with $|z(p)|, |z(q)|\le r$, assume the minimal geodesic $\gamma$ connecting $p, q$ passes through the origin $0\in D$. Viewed in terms of the coordinates $(w_1,z_2,\cdots,z_n)$ as above, in which $g_\alpha$ is flat and the two parts $\gamma|_{[p,0]}$ and $\gamma|_{[0,q]}$ intersect at $0$ with angle $\theta <\pi$. We know that $L_g(\gamma)\ge (1-\varepsilon(r))L_{g_\alpha}(\gamma)\ge (1-\varepsilon(r))(|w(p)| + |w(q)|)$, where $|w(p)|^2 = {|w_1(p)|^2+|z_2(p)|^2 + \ldots + |z_n(p)|^2}$ and similar for $|w(q)|$. On the other hand, the straight line $l_{p,q}$ (in the coordinates $\{w_1,z_2,\ldots,z_n\}$) connecting $p$ and $q$ has length $L_g(l_{p,q})\le (1+\varepsilon(r))L_{g_\alpha}{(l_{p,q})}= (1+\varepsilon(r)) \sqrt{|w(p)|^2+|w(q)|^2 - 2\cos \theta_{p,q} |w(p)||w(q)|}$. Since $\gamma$ is a minimal geodesic connecting $p$ and $q$, we have $L_g(\gamma)\le L_g(l_{p,q})$. Hence we get
\begin{equation}\label{eqn:cont}
(1-\varepsilon(r))(|w(p)|+ |w(q)|)\le (1+\varepsilon(r)) \sqrt{|w(p)|^2+|w(q)|^2 - 2\cos \theta_{p,q} |w(p)||w(q)|}.
\end{equation}
Note that \eqref{eqn:cont} also holds for any points $\tilde p\in \gamma|_{[p,0)}$ and $q\in \gamma|_{(0,q]}$, since $\gamma$ is also a minimal geodesic connecting $\tilde p$ and $\tilde q$. Take $\tilde p, \tilde q$ sufficiently close to $0$ such that $|w(\tilde p)| = |w(\tilde q)|$, then
\begin{equation*}
\frac{ \sqrt{|w(\tilde p)|^2+|w(\tilde q)|^2 - 2\cos \theta_{\tilde{p}, \tilde{q}} |w(\tilde p)||w(\tilde q)|}}{|w(\tilde p)|+|w(\tilde q)|}=\sqrt{ \frac{1-\cos\theta_{\tilde{p}, \tilde{q}}}{2}}<1,
\end{equation*}
while
\begin{equation*}
\frac{1-\varepsilon(r)}{1+\varepsilon(r)}\to 1,\quad\text{as } r\to 0.
\end{equation*}
Thus we would get a contradiction to \eqref{eqn:cont} if $r<<1$, that is, $\gamma$ cannot intersect with $D$.

\end{proof}
\begin{remark}\label{rem:tgc}
If $(X,D)$ is a toric manifold with normal crossing toric divisor $D=\sum_i D_i$, and $g$ is a smooth toric cone metric with angle $2\pi\alpha_i\in (0,2\pi)$ along each component $D_i$. Since around any point $p\in Supp D$, $g$ can be extended smoothly to a smooth metric in the $\alpha$-covering (see section \ref{sec:toric}), we don't need to approximate cone metric by the model metric as in Lemma \ref{lemma:geod}, indeed,  similar method as in the orbifold case \cite{Bo} can be applied to  give that $(X\backslash Supp D,g)$ is geodesic convex.
\end{remark}

\subsubsection{Volume comparison}
Assumptions as above, observe that
\begin{equation*}
\mathrm{Vol}_g(D) = 0.
\end{equation*}
To see this, we can take a tubular neighborhood of $D$, say, locally $D\subset \{z:|z_1|\le \varepsilon\}=:\Sigma_\varepsilon$. Since the metric $g$ is equivalent to $g_\alpha$, then
\begin{equation*}
\mathrm{Vol}(D)\le \mathrm{Vol}(\Sigma_\varepsilon)=\int_{\Sigma_\varepsilon} {\det g}\le C \int^\varepsilon_0r^{2\alpha-2}rdr\le C \varepsilon ^{2\alpha}\to 0,\quad \text{as }\varepsilon\to 0.
\end{equation*}
Hence we can define the volume of a set $A\subset X$ to be
\begin{equation*}
\mathrm{Vol}(A) := \int_{A\cap X\backslash D} \omega^n,
\end{equation*}
where $\omega$ is the K\"ahler form associated to $g$.

For a constant $k$, let $S^m_k$ denote the simply-connected $m$-dimensional space form of constant curvature $k$, $V_k(r)$ the volume of geodesic ball in $S^m_k$ with radius $r$.

\begin{theorem}[Relative Volume Comparison]
Let $(X,D)$ be a compact complex manifold with smooth divisor $D$, $g$ be a smooth cone metric with cone angle $2\pi\alpha\in  (0, 2\pi)$ along $D$. Suppose $\mathrm{Ric}(g)\ge (2n-1)k$ in $X\backslash D$ for some constant $k$. Let $p, x\in X$ be fixed points, $\rho = d(p,x)$, $V(p,r) = \mathrm{Vol}_g(B(p,r))$. Then we have the following:
\begin{enumerate}
\item\label{item:cp1} The function $r\mapsto \frac{V(p,r)}{V_k(r)}$ is non-increasing and if $p\in D$ then $\frac{V(p,r)}{V_k(r)}\to \alpha$ as $r\to 0^+$.
\item\label{item:cp2} For $r_1\le r_2\le r_3$,
\begin{equation*}
\frac{V(p,r_3) - V(p,r_2)}{V_k(r_3) - V_k(r_2)}\le \frac{V(p,r_1)}{V_k(r_1)}.
\end{equation*}
\item\label{item:cp3} For $r_1\le r_2$,
\begin{equation*}
\frac{V(p,r_2)}{V_{k}(\rho+r_2)}\le \frac{V(x,r_1)}{V_k(r_1)}.
\end{equation*}
\item\label{item:cp4} If $r_2+r\le \rho$, then
\begin{equation*}
\frac{V(p,r_2) V_k(r)}{V_k(\rho+r_2) - V_k(\rho-r_2)}\le V(x,r).
\end{equation*}
\item \label{item:cp5} Let $p\in X\setminus D$, $A$ be a star-convex set centered at $p$, then for any $r_1\le r_2$, we have
\begin{equation*}
\frac{V(B(p,r_2\cap A)) - V(B(p,r_2)\cap A)}{V_k(r_2) - V_k(r_1)}\le \frac{V(\partial B(p,r_1))}{V_k(\partial B_{r_1}^k)}.
\end{equation*}
\end{enumerate}
\end{theorem}
\begin{proof}
For any $p\in D$, we can choose a sequence $\{p_i\}\subset X\backslash D$ such that $d(p,p_i)\to 0$, and it's clear that for any $r>0$, $V(p_i,r)\to V(p,r)$. Hence, we only need to prove the above inequalities for $p\in X\backslash D$. Note that the cone metric $g$ is smooth in $X\backslash D$(being incomplete), we can define the cut-loci of $p$ to be the same as in the smooth case, say, a point $q\not\in D$ is at the cut-loci of $p$, if either $q$ is a conjugate point of $p$, or there exist at least two different minimal geodesics connecting $p$ and $q$. Moreover, in this cone metric case, we also define $D$ to be contained in the cut-loci of $p$ because of Lemma \ref{lemma:geod}. Denote $CL(p)$ to be the cut-loci of $p$. Clearly $CL(p)$ has measure $0$. For any star-convex set $A\subset X$ (in the sense that any point $q\in A)$ can be joined to $p$ with the whole geodesic contained in $A$, in particular, $X\backslash CL(p)$ is a star-convex set).%, denote $\tilde A:= A\cap X\backslash CL(p)$.
In $A\cap X\setminus CL(p)$, the metric $g$ is smooth and the distance function $d(\cdot):=d(\cdot, p)$ is smooth except at $p$, and the metric has Ricci curvature bounded below by $(2n-1)k$. By relative comparison in Riemannian geometry, we can see that the function
\begin{equation*}
r\mapsto \frac{\mathrm{Vol}(\partial B(p,r)\cap A)}{V_k(\partial B^k_r)}=:f(r)
\end{equation*}
is non-increasing, where $B^k_r$ is a geodesic ball in $S^m_k$ with radius $r$ and $m=2n$. Then the proof of the items except the second part in (1) goes similar as that in \cite{CGT} by taking $A= X\backslash CL(p)$.

For the second part of (\ref{item:cp1}), viewed in the coordinates $\{w_1,z_2,\ldots,z_n\}$, the metric $g$ can be viewed as a metric with Ricci curvature bounded below by $(2n-1)k$ in $S\times \mathbb C^{n-1}$, where $S$ is a fan-shape space with angle $2\pi\alpha$ in  $\mathbb R^2\cong \mathbb C$. Precisely, $\arg w_1 = \alpha \arg z_1$, thus  for $p\in D$
\begin{equation*}
\frac{V(B(p,r))}{V_k(r)}\to \alpha,
\end{equation*}
as $r\to 0$, which can be seen by writing the volume in terms of polar coordinates.

\end{proof}
If the cone metric satisfies $\mathrm{Ric}(g)\ge (2n-1)k>0$, then we have the diameter bound.
\begin{theorem}[Myers Theorem]
Let $(X, D, g)$ be a cone metric with cone angle $2\pi\alpha\in (0, 2\pi)$ along $D$, and $\mathrm{Ric}(g)\ge (2n-1)k>0$, then the diameter of $(X,g)$ is bounded above by $\pi/\sqrt{k}$.
\end{theorem}
\begin{proof}
Note that $g$ is smooth in $X\backslash D$ and its Ricci curvature is bounded below by $(2n-1)k>0$, and any two points $p,q\in X\backslash D$ can be joined by a minimal geomdesic which has no intersection with $D$, thus the same proof as that in the classic Myers Theorem applies to this case, and we have $\mathrm{diam}(X\backslash D,g)\le \pi/\sqrt{k}$. For any points $p, q\in D$, there exist $\{p_i\}, \{q_i\}\subset X\backslash D$ such that $d(p,p_i)\to 0$ and $d(q,q_i)\to 0$. Then triangle inequality implies $d(p,q)\le \lim d(p_i,q_i) + d(p,p_i) + d(q,q_i)\le \pi/\sqrt{k}$. Thus $\mathrm{diam}(X,g)\le \pi/\sqrt{k}$.
\end{proof}
Next we state a lemma due to Gromov \cite{Gr}.

\begin{lemma} \label{grom}
Suppose $g$ is a cone metric on $(X,D)$ with cone angle $2\pi\alpha\in (0,2\pi)$, and $\mathrm{Ric}(g)\ge (2n-1)k>0$. Let $E$ be a small tubular neighborhood of $D$. Let $\varepsilon>0$ and $p_1$, $p_2$ be two points in $X\backslash E$ such that $B(p_i,\varepsilon)\cap E = \emptyset$ ($i=1,2$).If for any point $p\in B(p_2,\varepsilon)$, there is a minimal geodesic connecting $p_1$ and $p$ and intersecting with $\partial E$, then there exists a constant $C=C(n,\varepsilon,k)$ such that
\begin{equation*}
\mathrm{Vol}(B(p_2,\varepsilon))\le C \mathrm{Vol}(\partial E).
\end{equation*}
\label{gromov's lemma}
\end{lemma}
\begin{proof}
We proceed the proof following Gromov \cite{Gr}. For any $p\in B(p_2,\varepsilon)$, let $\gamma_p$ be a minimal geodesic connecting $p_1$ and $p$, by assumption, $\gamma_p\cap \partial E\neq \emptyset$, denoted by  $\tilde p\in \gamma\cap \partial E$ (if there are more than one intersection points, we take the one closest to $p$). Set $d_1(\gamma_p):= d(p_1,\tilde p)$ and $d_2(\gamma_p): =d(p,\tilde p)$. Then clearly, $\varepsilon<d_i(\gamma_p)\le D$, $i=1,2$, and here $D$ is the diameter of $(X,g)$, which is bounded above by $\pi/\sqrt{k}$.  Applying the relative volume comparison theorem (the $5^{\text{th}}$- item), we have
\begin{equation*}
\frac{\mathrm{Vol}(B(p_2,\varepsilon))}{\mathrm{Vol}(\partial E)}\le \sup_\gamma \frac{V_k(d_1(\gamma) + d_2(\gamma)) - V_k(d_1(\gamma))}{V_k(\partial B_k(d_1(\gamma)))}\le \frac{V_k(D) - V_k(\varepsilon)}{V_k(\partial B_k(\varepsilon))}=: C(n,k,\varepsilon).
\end{equation*}
\end{proof}
\begin{remark}
Although we only state the volume comparison theorem, Myers theorem and Gromov's lemma, for smooth divisor case, they can be easily generalized for smooth toric conical K\"ahler metrics.  We will need the geodesic convexity as in Remark \ref{rem:tgc}.
\end{remark}

\section{Toric K\"ahler-Ricci solitons with conical singularities}

In this section, we will prove Theorem \ref{theorem-1} and Theorem \ref{theorem-2}. The key ingredient is to prove a  family version of $C^0$ estimates in \cite{WZ} which will help us prove Theorem \ref{theorem-3} for a possibly degenerate family of toric manifolds.

\subsection{Setting up the continuity method}

\noindent For the rest of this section, $X$ will be a toric manifold with a fixed polarization by an ample line bundle $L$. We borrow notation from the last section. In particular, recall that the polytope given by $L$ is fixed to be
\begin{equation*}
P = \{x \in \mathbb{R}^n ~|~ l_j(x) = v_j \cdot x + \lambda_j > 0\}.
\end{equation*}

Also, we once and for all fix a reference metric in $c_1(L)$ by setting $\homega = \ddbar\hat \varphi$ on $(\mathbb{C}^*)^n$ where,

\begin{equation}
\hat u(x) = \sum_{j=1}^{N}{\frac{l_j(x) \log {(l_j(x))}}{\be_j}}
\end{equation}

\noindent and $\hvarphi$ is the Legendre transform of $\hat u$.  Then, by the discussion in the last section, $\homega$ is a global toric smooth conical metric with angles $2\pi\be_j$ along the divisor $D_j$.

Our aim in this section is to solve the following conical soliton equation
\begin{equation*}\tag{*}
Ric(\omega) = \al\omega + \mathcal{L}_\xi\omega + [D],
\label{soliton eq}
\end{equation*}
where $\alpha>0$, $\omega$ is a smooth toric conical K\"ahler metric, $\xi$ is a holomorphic toric vector field on $X$ and $D$ is an effective toric $\mathbb{R}$-divisor. This is a generalization of Wang-Zhu \cite{WZ} in the case of smooth Fano manifolds. We will prove our estimates in the framework of \cite{WZ} combined with some techniques of \cite{D2}. On the open part $(\mathbb{C}^*)^n$, we can write $\omega = \ddbar\varphi$. $\xi$ being holomorphic then implies that $\mathcal{L}_\xi\omega = \partial\bar\partial\xi(\varphi)$. Since $\xi$ is also toric , it is generated by the standard vector fields $\{z_i\partial/\partial z_i\}$. Consequently, there exists a vector $\vec{c} \in \mathbb{R}^n$ such that

\begin{equation*}
\xi(\phi) = \sum_{i=1}^{n}{c_i\frac{\partial \varphi}{\partial \rho_i}}.
\end{equation*}

Since on the open part one does not see the divisor, the soliton equation can be re-written as a real Monge-Ampere equation -

\begin{equation}
\det{(\nabla^2 \varphi)} = e^{-\al \varphi - c\cdot \nabla\varphi + \al\tau\cdot\rho}
\label{soliton real m.a}
\end{equation}

\noindent for some $\tau \in \mathbb{R}^n$. Here the linear part shows up when one gets rid of the $\partial\bar\partial$ and in some sense corresponds to the divisor and controls the blow up of the metric as is seen below.

\begin{lemma}\label{sani}
If there exists a solution of \eqref{soliton real m.a} then $\tau$ and the vector $\vec{c}$ must satisfy

\begin{equation}
\tau = \frac{\int_{P}{xe^{c\cdot x} \,dx}}{\int_{P}{e^{c\cdot x} \,dx}} .
\end{equation}

\noindent Moreover, the divisor $D$ in $\eqref{soliton eq}$ is given by

\begin{equation*}
D = \sum_{j}{(1- \al l_j(\tau)) D_j}
\end{equation*}

\noindent and so the cone angle along each $D_j$ is $2\pi\be_j$ where $\be_j = \al l_j(\tau)$.
\label{cone angles}
\end{lemma}

\begin{proof} Since $\det{(\nabla^2 \varphi)} \rightarrow 0$ as $|\rho| \rightarrow \infty$, for $\al > 0$,
\begin{align*}
0 &= \int_{\mathbb{R}^n}{\nabla(e^{-\al \varphi + \al\tau\cdot\rho})\,d\rho} \\
%&= \int_{\mathbb{R}^n}{\nabla(-\al\varphi + \al\tau\cdot\rho)e^{c\cdot \nabla\varphi}\det(\nabla^2\varphi) \,d\rho} \\
%&= \al\tau \int_{\mathbb{R}^n}{e^{c\cdot \nabla\varphi}\det(\nabla^2\varphi) \,d\rho} - \al \int_{\mathbb{R}^n}{\nabla\varphi e^{c\cdot \nabla\varphi}\det(\nabla^2\varphi) \,d\rho}   \\
&= \al\tau \int_{P}{e^{c\cdot x} \,dx} - \al\int_{P}{xe^{c\cdot x} \,dx}.
\end{align*}

To compute the cone angles we consider the asymptotics at infinity. First, the equation for the symplectic potential $u$, the Legendre transform of $\varphi$, is given by

\begin{equation}
\det(\nabla^2 u) = e^{-\al u + \al(x - \tau ) \cdot\nabla u - c\cdot x } .
\label{symp. eqn.}
\end{equation}

By the conic Gullemein boundary conditions,

\begin{equation*}
u = \hat u + f(x) =  \sum_{j=1}^{N}{\frac{l_j(x) \log {(l_j(x))}}{\be_j}} + f(x)
\end{equation*}

\noindent for some $f \in C^{\infty}(\bar P)$. By direct computation it can be seen that

\begin{equation*}
\det(\nabla^2 u) = \frac{G(x)}{l_1(x) \dotsc l_N(x)}
\end{equation*}

\noindent for some non vanishing $G \in C^{\infty}(\bar P)$. On the other hand, once again using the formula for $u$, the order of $l_j$ on the right hand side of equation \eqref{symp. eqn.} can be seen to be $\al l_j(\tau)/\be_j$. Comparing the orders of $l_j(x)$ on both sides of the equation we conclude that $\be_j = \al l_j(\tau)$.

\end{proof}

Hence $D$ is effective if and only if $1-\alpha l_j(\tau)>0$ for all $j$ and in this case Lemma \ref{sani} implies that $\tau \in P$. Conversely, we have the following Lemma due to Wang-Zhu and Donaldson \cite{WZ}, \cite{D2}

\begin{lemma}
For each $\tau \in P$, there exists a unique vector $\vec{c} \in \mathbb{R}^n$ satisfying

\begin{equation}
\tau = \frac{\int_{P}{xe^{c\cdot x} \,dx}}{\int_{P}{e^{c\cdot x} \,dx}}.
\label{compatible}
\end{equation}
\end{lemma}

\begin{proof}
By translating the polytope by $\tau$, we can assume without loss of generality that $\tau = 0$. Consider the function

\begin{equation*}
F(\vec{c}) = \int_{P}{e^{c\cdot x} \,dx}.
\end{equation*}

Clearly this function is strictly convex as can be seen by differentiating it twice. It is also proper. This follows from $0$ being an interior point. Hence the function has a unique minimum $\vec{c}$. But then $\nabla F(\vec{c}) = 0$ which is precisely what we need.

\end{proof}

By the Cartan formula, for any K\"ahler metric $\omega$, $\mathcal{L}_{\xi}\omega = di_{\xi}{\omega} $. Since $\xi$ is holomorphic, clearly $\bar\partial i_{\xi}\omega = 0$. Now, all toric manifolds are simply connected i.e $H^{0,1}(X,\mathbb{C}) = 0$. So there exists a potential function $\theta_{\xi}$ such that $\xi = \nabla \theta_{\xi}$. Of course the function also depends on the metric. The Lie derivative is now given by
\begin{equation}
\mathcal{L}_{\xi}\omega = \ddbar\theta_{\xi}.
\end{equation}
From now on, we fix $\tau \in P$ with $1- \al l_j(\tau)>0$ for all $j$ and $\xi$ is the unique holomorphic vector field determined by $\tau$ as in Lemma \ref{sani}.

For the continuity method, we need to set up the Monge-Ampere equation. For that we need an analogue of the $\partial\bar\partial$-lemma in this conical setting. We also set $\be(\al) = (\al l_1(\tau),\dotsc , \al l_N(\tau))$.
By lifting the smooth conical K\"ahler metric $\hat\omega$ to each uniformization covering, we can obtain the following lemma.

\begin{lemma}
There exists a unique (up to constants) function $h \in C^{\infty}_{\be(\al)}(X)$ satisfying
\begin{equation}
Ric(\homega) - \al \homega - [D] - \mathcal{L}_{\xi}\homega = \ddbar h.
\label{ddbar lemma}
\end{equation}
\end{lemma}

%\begin{proof}
%A function $h$ on $(\mathbb{C}^*)^n$ extends to a function in $C^{\infty}_{\be(\al)}(X)$ if and only if its composition with $\nabla \hat u$ extends to a smooth (in the usual sense) function on $\overline P$. For notational convenience we also denote this composition by $h$. On the open part the equation is equivalent to

%\begin{equation*}
%det(\nabla^2 \hvarphi) = e^{-\al (\hvarphi - \tau \cdot \rho) - c\cdot \nabla\hvarphi - h}
%\end{equation*}

%\noindent In symplectic coordinates this translates to

%\begin{align*}
%det(\nabla^2 \hat u) &= e^{-\al \hat u + \al(x-\tau)\cdot \nabla \hat u - c \cdot x - h } \\
%\end{align*}

%\noindent We also note that as before

%\begin{equation*}
%det(\nabla^2 u) = \frac{G(x)}{l_1(x) \dotsc l_N(x)}
%\end{equation*}

%\noindent where $G$ is smooth function on $P$ bounded away from zero. Using this and the fact that $\be_j = \al l_j(\tau)$, after solving for $h$ in the above equation, we can see that $h$ extends to a smooth function on $\overline P$. Hence $h$, seen now as a function on $(\mathbb{C}^*)^n$, extends to a function in $C^{\infty}_{\be(\al)}(X)$

%\end{proof}

We now write the Monge-Ampere equation for the conical soliton. Set $\omega = \homega + \sqrt{-1}\partial\bar\partial \psi$. Then the equation for the conical soliton is
\begin{equation} \tag*{$(**)$}
\begin{cases}
(\homega + \sqrt{-1}\partial\bar\partial \psi)^n = e^{-\al\psi - \xi(\psi) + h}\homega^n\\
\homega + \sqrt{-1}\partial\bar\partial \psi > 0 \\
\psi \in C^{\infty}_{\be(\al)}(X),
\end{cases}
\label{ma}
\end{equation}
where $h$ is from the above lemma.

To solve this equation, like usual, we introduce a parameter $s \in [0,\al]$ and look at the following family of equations.
\begin{equation} \tag*{$(**)_s$}
\begin{cases}
(\homega + \sqrt{-1}\partial\bar\partial \psi_s)^n = e^{-s\psi_s - \xi(\psi_s) + h}\homega^n \\
\homega + \sqrt{-1}\partial\bar\partial \psi_s > 0 \\
\psi_s \in C^{\infty}_{\be(\al)}(X)
\end{cases}
\label{ma s}
\end{equation}
or equivalently
\begin{equation} \tag*{$(**)_s$}
Ric(\omega_s) = s \omega_s + (\al-s)\homega + \mathcal{L}_{\xi}\omega_s + D.
\label{eqn. s}
\end{equation}

The corresponding linearized operator is given by
\begin{equation*}
L_s(\psi) = L(\psi) + s\psi= \Delta \psi + \xi(\psi)+ s\psi.
\end{equation*}
Recall that we are only looking at the space of functions invariant under the toric action. One can define an inner product by
\begin{equation*}
(\psi_1,\psi_2) = \int_{X}{\psi_1 \bar \psi_2 e^{\theta_\xi} \omega^n}
\end{equation*}
and denote the corresponding Hilbert space of square integrable functions by $L^2(e^{\theta_\xi} )$. Then $L$ restricted to $C^{\infty}_{\be}(X)$ is self adjoint and hence can be thought of as an operator from $L^2(e^{\theta_\xi})$ to itself. Also, by virtue of being self adjoint, $L$ only has real eigenvalues. The linear theory for the spaces $C^{k,\gamma}_{\be}(X)$ is summarized below.

\begin{lemma}
Let $\omega$ be a $\be$-conical metric, $\Delta$ be the corresponding Laplacian and $L$ be defined as above. Then

\begin{itemize}
\item[(1)] For $k \geq 2$,  $\Delta : C^{k,\gamma}_{\be}(X) \rightarrow C^{k-2,\gamma}_{\be}(X)$ is an invertible operator, modulo constants
\item[(2)] The Fredholm alternative holds for $L$.
\item[(3)] All nonzero eigenvalues of $-L$ are positive. Moreover, if $Ric(\omega) > t\omega + \mathcal{L}_{\xi}\omega$ and $-L\psi = \lambda \psi$, then $\lambda > t$.
\end{itemize}
 \label{linear theory}
\end{lemma}

The lemma that follows is essentially an observation of Zhu \cite{Zhu}, adapted to the conical setting and is required in all the subsequent estimates. The proof in the toric case is in fact much easier.

\begin{lemma}\label{zhus lemma}

There exists a uniform constant $C$ depending only on $\homega$ and $\xi$ such that, for any function $\psi \in C_{\be}^{\infty}(X) \cap PSH(X,\homega)$ %satisfying $\omega = \homega + \sqrt{-1}\partial\bar\partial \psi > 0$,
\begin{equation*}
|\xi(\psi)| \leq C.
\end{equation*}
\end{lemma}
\begin{proof}
In the toric situation, the proof is almost trivial. Locally on $(\mathbb{C}^*)^n$, $\homega = \partial\bar\partial \hvarphi$ and for any $\varphi = \hvarphi + \psi$, since $\psi$ is globally bounded and plurisubharmonic, it is easy to see that $\nabla \varphi (\mathbb R^n) = \nabla \hvarphi (\mathbb R^n) = P$ and $\partial\bar\partial \varphi$ extends to a global K\"ahler metric.  So there exists a uniform constant $C$ such that
\begin{equation*}
|\nabla \psi| \leq C.
\end{equation*}
\noindent But then, since $\xi$ is given by
\begin{equation*}
\xi = \sum_{i=1}^{n}{c_i z_i\frac{\partial}{\partial z_i}},
\end{equation*}
\noindent we have that

\begin{equation*}
\xi(\psi) = c\cdot \nabla \psi.
\end{equation*}

\noindent This gives us the required bound.

\end{proof}

The following proposition shows that there exists a solution to equation \ref{ma s} at $s=0$

\begin{proposition}
For any $h \in C^{\infty}_{\be(\al)}(X)$ there exists a unique function $\psi \in C^{\infty}_{\be(\al)}(X)$ satisfying
\begin{equation*}
\begin{cases}
(\homega + \sqrt{-1}\partial\bar\partial \psi)^n = e^{h - \xi(\psi)}\homega^n \\
\sup{\psi} = 0\\
\omega = \homega + \sqrt{-1}\partial\bar\partial \psi > 0.
\end{cases}
\label{calabi-yau}
\end{equation*}

\end{proposition}

\begin{proof}
We proceed by the continuity method. Consider the family of equations
\begin{equation}
\begin{cases}
(\homega + \sqrt{-1}\partial\bar\partial \psi_s)^n = e^{sh - \xi(\psi_s)}\homega^n \\
\sup{\psi_s} = 0\\
\omega = \homega + \sqrt{-1}\partial\bar\partial \psi_s > 0
\end{cases}
\label{random eq}
\end{equation}
\noindent and set $S = \{s \in [0,1] |$ equation \eqref{random eq} has a solution $\psi_s \in C^{3,\gamma}_{\be(\al)}(X)$ at $s$\}. The set $S$ is clearly nonempty, since $0 \in S$. In what follows we suppress the index $s$ for convinience.
\medskip

\noindent \textit{Openness.} This follows straight from part(a) of Lemma \ref{linear theory} and the implicit function theorem on the space $C^{\infty}_{\be(\al)}(X)$, since the linearized operator is just $L$.

\medskip

\noindent \textit{$C^0$ estimates.} By Lemma \ref{zhus lemma}, the right hand side of the equation is uniformly bounded in $s$ and hence in particular there exists a uniform $L^p$ bound for any $p>1$. Now, Kolodziej's results  and their generalizations \cite{Ko, EGZ, ZhZ} give a uniform $C^0$ bound.

\medskip

\noindent \textit{Second order estimates.} Consider the quantity

\begin{equation*}
H_s = \log{tr_{\homega}\omega_s} -A\psi,
\end{equation*}

\noindent where $A$ is some large number to be chosen later. Since both $\homega$ and $\omega_s$ have poles of same order, the quantity is bounded. Let $\sup_{X}{H_s} = H_s(q)$. We lift all the local calculations to the $(S^1)^n$ invariant $\be$- covering space. The second order estimates easily follow from \cite{Y1} and \cite{TZhu}.

\medskip
\noindent \textit{$C^3$ and higher order estimates.} Calabi's method for third order estimates can again be carried out by lifting the calculations to the $\be$-cover. The reader should refer to \cite{PSS} for the simplified computations. Higher order derivatives can be obtained by a standard bootstrapping argument. Closedness now follows from Ascoli-Arzela. Hence $1 \in S$.

\end{proof}

\subsection{Estimates and proofs of Theorems \ref{theorem-1} and \ref{theorem-2}} For later applications, we need to get the precise dependence of the $C^0$ estimate on the polytope. So we introduce some notation. Recall that \\
\centerline{$P = \{x \in \mathbb{R}^n~ |~ l_j(x) = v_j \cdot x + \lambda_j > 0\}$.}

\medskip

\noindent We let $\nu$ and $\sigma$ be two constants such that
\begin{align*}
\nu^{-1}<Vol(P) &< \nu,  \\
(\sigma)^{-1} < diam (P)& < \sigma.
\end{align*}

On $(\mathbb{C}^*)^n$ we write $\homega = \ddbar \hvarphi$ and $\omega_s = \ddbar\varphi_s$. Using the standard logarithmic coordinates like before one can rewrite equation \ref{ma s} as a real Monge-Ampere on $\mathbb{R}^n$
\begin{equation}
\begin{cases}
\det(\nabla^2 \varphi_{s}) = e^{-s(\varphi_{s} - \tau\cdot \rho) - (\al - s)(\hvarphi - \tau\cdot\rho) - c\cdot \nabla \varphi_s}, \\
u_{s} = \mathcal{L}\varphi_{s} = \sum_{j=1}^N(\beta_j)^{-1} l_j(x) \log l_j(x) + f (x), \hspace{0.5in} f \in C^{\infty}(\overline P)\\
\nabla^2 \varphi_{s} > 0.
\end{cases}
\end{equation}

\begin{proposition}
For any $s_0 \in (0,\al)$ there exists a constant $C = C(n,s_0, \nu, \sigma, \Lambda, \sup_{j,P}{|l_j|})$  such that

\begin{equation*}
|\varphi_s - \hvarphi| \leq C
\end{equation*}
\noindent for all $s \in [s_0, \al)$. Here $\Lambda$ is the constant from Lemma \ref{ggrad} below.
\label{c^0 estimate}
\end{proposition}

We use the arguments  in \cite{WZ} with inputs and simplifications from  \cite{D2}, most notably the last step which helps us avoid the Harnack inequality. We first need two technical lemmas.

 \begin{lemma}
 Suppose $v \geq 0$ is a strictly convex function on $\mathbb{R}^n$ such that $v(0) = 0$ and $\det(\nabla^2 v) \geq \lambda $ on $v \leq 1$. Then there exists $ C > 0$ such that
 \begin{equation}
Vol(v \leq 1) \leq  C\lambda^{-1/2}.
\end{equation}
 \label{conv anal lemma}
  \end{lemma}

 \noindent The proof is a standard barrier function argument and so we skip it.

\begin{lemma} \label{ggrad}
\noindent If $\hvarphi$ is the Legendre transform of $\hat u$ and we define the function
\begin{equation}
g_j(\rho) = \log(l_j(\nabla \hvarphi (\rho))).
\end{equation}
\noindent Then, there exists a constant $\Lambda$ such that
\begin{equation}
 \sup_{\mathbb{R}^n}{|\nabla g_j|} \leq \Lambda.
\end{equation}
\noindent Here, $\Lambda$ depends only on $\be_j$, $N$, $n$ and the normal vectors $v_j$.
\label{donaldson lemma}
\end{lemma}

\begin{proof}

Recall that the polytope $P$ is given by faces $l_j(x) = v_j \cdot x + \lambda_j$ and let $\hat u$ be the usual conical symplectic potential given by
\begin{equation}
\hat u = \sum_{j=1}^{N}{\frac{1}{\be_j}l_j \log l_j}.
\end{equation}
We then set,
\begin{equation*}
V = \{v_{j\ga}\}, \hspace{0.5in}  A = \{\frac{v_{j\ga}}{\sqrt{\be_jl_j}}\} := \{a_{j\ga}\},
\end{equation*}
\noindent where $v_j = \{v_{j\ga}\}$ is the vector normal to the face $l_j$. For any  $J = \{j_1, \dotsc j_n \} \subset \{1, \dotsc N\}$ in some order, we let $M_J$ be the corresponding $n \times n$ minor of $V$ and $C_J = det(M_J)$.  Let $\tilde M_J$ and $\tilde C_J$ be the corresponding quantities for $ A$. Then, in our notation

\begin{equation}
\nabla^2 \hat u := \{\hat {u}_{\ga\mu}\} = ( A)^t  A .
\end{equation}

\noindent \textit{\textbf{Claim 1}}  \[\det(\nabla^2 \hat u) = \sum_{j_1<\dotsc < j_n}{\tilde {C}_{j_1,\dotsc j_n}^2} = \sum_{j_1< \dotsc <j_n}{\frac{C^2_{j_1,\dotsc j_n}}{(\be_{j_1}l_{j_1})\dotsc (\be_{j_n}l_{j_n})}}\] .\\§
\noindent This is known as the Cauchy-Binet formula in literature. The proof is of course just a simple exercise in undergraduate  linear algebra and so we skip it. \\

\noindent  Now, for $J = \{j_1, \dotsc, j_{n-1}\}$ let $M_{J ;\ga}$ be the minor obtained by deleting the $\ga^{th}$ column from the matrix of row vectors $v_{j_i}$. We once again set $C_{J ; \ga} = det(M_{J ;\ga})$\\

\noindent \textit{\textbf{Claim 2}}
\begin{align}
\det(\nabla^2 \hat u) \sum_{\mu=1}^{n}{a_{j\mu} \hat {u}^{\mu\ga}} &= (-1)^{\ga+1}\sum_{\substack{j_1<\dotsc < j_{n-1} \\ j_i \neq j}}{\tilde C_{j,j_1,\dotsc j_{n-1}}\tilde C_{j_1,\dotsc j_{n-1} ; \ga}}\\
& = \frac{(-1)^{\ga+1}}{\sqrt{\be_{j}l_j}} \sum_{\substack{j_1<\dotsc < j_{n-1} \\ j_i \neq j}}{\frac{ C_{j,j_1,\dotsc j_{n-1}} C_{j_1,\dotsc j_{n-1} ; \ga}}{(\be_{j_1}l_{j_1})\dotsc (\be_{j_{n-1}}l_{j_{n-1}})}},
\end{align}

\noindent where $\{u^{\mu\ga}\}$ denotes the inverse matrix of $\nabla^2 \hat u$.

\begin{proof}The proof proceeds along the lines of the proof for Cauchy-Binnet formula, only it requires more book keeping and is as follows -  Denoting by $\chi_{\ga\mu}$, the co-factor matrix of $\nabla^2 u$ and employing Cramer's rule,
\begin{align*}
\det(\nabla^2 \hat u) \sum_{\mu=1}^{n}{a_{j\mu} \hat {u}^{\mu\ga}} &=  \sum_{\mu=1}^{n}{a_{j\mu}\chi_{\ga\mu}}\\
&= \sum_{\mu=1}^{n}{a_{j\mu}}\sum_{\substack{\sigma : \{1,\dotsc,\hat \mu, \dotsc ,n\} \\ \rightarrow \{1, \dotsc, \hat \ga, \dotsc n\}}}{(-1)^{\ga + \mu}sgn(\sigma)u_{1\sigma(1)}\dotsc u_{l-1 \sigma(l-1)}u_{l+1 \sigma(l+1)} \dotsc u_{n\sigma(n)}}\\
&=  \sum_{\mu=1}^{n}{a_{j\mu}} \sum_{j_1 =1}^{N} \dotsc \sum_{j_{n-1} = 1}^{N}\sum_{\substack{\sigma : \{1,\dotsc,\hat \mu, \dotsc ,n\} \\ \rightarrow \{1, \dotsc, \hat \ga, \dotsc n\}}}{(-1)^{\ga+\mu}sgn(\sigma)a_{j_1 1} a_{j_1 \sigma(1)}\dotsc a_{j_{n-1} n} a_{j_{n-1} \sigma(n)}} ,
\end{align*}

\noindent where the product in the above summation includes exactly two entries from all columns except the $\mu^{th}$ and the $\ga^{th}$ ones which have one entry each. Clearly, the innermost summation, as $\sigma$ runs over all permutations, is some determinant. More precisely,
\begin{align*}
\det(\nabla^2 \hat u) \sum_{\mu=1}^{n}{a_{j\mu} \hat {u}^{\mu\ga}} &=    \sum_{\mu=1}^{n}{a_{j\mu}} \sum_{j_1 =1}^{N} \dotsc \sum_{j_{n-1} = 1}^{N} {(-1)^{\ga+\mu}a_{j_1 1}\dotsc a_{j_{n-1}n}\tilde C_{j_1,\dotsc j_{n-1} ; \ga}}   .
\end{align*}
Again, like in the proof of the first claim, $\tilde C_{j_1,\dotsc j_{n-1} ; \ga} = 0$ unless the $j_i$'s are distinct. Also, if $j_1 < \dotsc j_{n-1}$ and $\tau$ any permutation of this indices, then $\tilde C_{\tau(j_1),\dotsc \tau(j_{n-1}) ; \ga} = \tilde sgn(\tau)C_{j_1,\dotsc j_{n-1} ; \ga}$. Thus,
\begin{align*}
det(\nabla^2 \hat u) \sum_{\mu=1}^{n}{a_{j\mu} \hat {u}^{\mu\ga}} &= \sum_{\mu=1}^{n}{a_{j\ga}} \sum_{j_1<\dotsc < j_{n-1}}\sum_{\tau}{(-1)^{\ga+ \mu}sgn(\tau)a_{\tau(j_1) 1}\dotsc a_{\tau(j_{n-1})n}\tilde C_{j_1,\dotsc j_{n-1} ; \ga}}\\
&=  \sum_{\mu=1}^{n}{a_{j\mu}} \sum_{j_1<\dotsc < j_{n-1}}\sum_{\substack{\tau: \{1,\dotsc ,\hat \mu, \dotsc n\}  \\ \rightarrow \{1,\dotsc ,\hat \mu, \dotsc n\}}}{(-1)^{\ga+\mu}sgn(\tau)a_{j_1 \tau(1)}\dotsc a_{j_{n-1}\tau(n)}\tilde C_{j_1,\dotsc j_{n-1} ; \ga}}\\
&= \sum_{j_1<\dotsc < j_{n-1}}\sum_{\mu=1}^{n}{(-1)^{\ga+\mu}a_{j\mu}\tilde C_{j_1,\dotsc j_{n-1} ; \mu} \tilde C_{j_1,\dotsc j_{n-1} ; \ga}}\\
&= (-1)^{\ga+1}\sum_{j_1<\dotsc < j_{n-1}}{\tilde C_{j,j_1,\dotsc j_{n-1}} \tilde C_{j_1,\dotsc j_{n-1} ; \ga}} .
\end{align*}

\end{proof}

\textit{\textbf {Proof of Lemma \ref{ggrad}}} We compute the derivative of $g_j$ using the correspondence $\nabla ^2 \hvarphi = (\nabla^2 \hat u)^{-1}$ and $x = \nabla \hvarphi$.
\begin{align*}
\frac{\partial g_j}{\partial \rho_\ga} &= \frac{l_j(\nabla(\varphi_\ga))}{l_j(\nabla \varphi)} = \frac{1}{l_j(\nabla \varphi)}\sum_{\mu=1}^{n}{v_{j\mu} \hat{u}^{\mu\ga}}= \frac{\sqrt{\be_j}}{\sqrt{l_j}}\sum_{\mu=1}^{n}{a_{j\mu} \hat{u}^{\mu\ga}}\\
&=  \frac{1}{det(\nabla^2 \hat u)}\sum_{\substack{j_1<\dotsc < j_{n-1} \\ j_i \neq j}}{\frac{1}{l_j}\frac{ C_{j,j_1,\dotsc j_{n-1}} C_{j_1,\dotsc j_{n-1} ; \ga}}{(\be_{j_1}l_{j_1})\dotsc (\be_{j_{n-1}}l_{j_{n-1}})}} \\
&= \frac{\sum_{\substack{j_1<\dotsc < j_{n-1} \\ j_i \neq j}}{\frac{1}{l_j}\frac{ C_{j,j_1,\dotsc j_{n-1}} C_{j_1,\dotsc j_{n-1} ; \ga}}{(\be_{j_1}l_{j_1})\dotsc (\be_{j_{n-1}}l_{j_{n-1}})}}}{\sum_{j_1< \dotsc <j_n}{\frac{C^2_{j_1,\dotsc j_n}}{(\be_{j_1}l_{j_1})\dotsc (\be_{j_n}l_{j_n})}}}\\
&\leq \frac{\sum_{\substack{j_1<\dotsc < j_{n-1} \\ j_i \neq j}}{\frac{1}{l_j}\frac{ C_{j,j_1,\dotsc j_{n-1}} C_{j_1,\dotsc j_{n-1} ; \ga}}{(\be_{j_1}l_{j_1})\dotsc (\be_{j_{n-1}}l_{j_{n-1}})}}}{\sum_{j_1< \dotsc <j_{n-1}}{\frac{C^2_{j,j_1,\dotsc j_{n-1}}}{(\be_jl_j)(\be_{j_1}l_{j_1})\dotsc (\be_{j_{n-1}}l_{j_{n-1}})}}} .
\end{align*}
\noindent The summation can be taken to be only over all $J = \{j_1,\dotsc j_{n-1}\}$ such that $C_{jJ} \neq 0$. For such terms, one has the trivial bound
\begin{equation*}
\left | \frac{C_{j_1,\dotsc j_{n-1} ; \ga}}{C_{j,j_1,\dotsc j_{n-1}}} \right | \leq M',
\end{equation*}

\noindent where $M'$ depends only on upper bounds on $|v_j|$ and $\be_j$ and a positive lower bound on $|C_J|$ as $J \subset \{1, \dotsc N\}$ varies over all subsets with $C_J \neq 0$. Together with the above computation, we get
\begin{equation*}
 \left | \frac{\partial g_j}{\partial \rho_\ga} \right | \leq \be_j \Lambda' .
\end{equation*}

\noindent This proves the Lemma with $\Lambda = \be_j \Lambda'$.

\end{proof}

\noindent \textbf{Proof of Proposition \ref{c^0 estimate}} There are several steps following \cite{WZ} and \cite{D2} combined with Lemma \ref{ggrad}. Let $\phi_{s} = \varphi_{s} - \tau\cdot \rho$, $\hat\phi = \hvarphi- \tau\cdot\rho$ and define
 \begin{equation}
 w_{s} = s\phi_{s} + (\al - s)\hat\phi .
 \label{w}
 \end{equation}
 \noindent Set \\
 \centerline{$m_{s} := \inf_{\mathbb{R}^n}w_{s} = w_{s}(\rho_{s})$}

 \medskip

\noindent \textbf{Step 1.} We claim that there exist $ C, \zeta >0$ independent of $s$ such that for all $s \in [s_0, \al]$,

\begin{align}
&(a) ~|m_{s}| \leq C \\
&(b)~ w_{s} \geq \zeta|\rho - \rho_{s}| - C . \label{key estimate}
\end{align}

It follows from the definition of $w_{s}$ that $\det(\nabla^2 w_{s}) \geq s^n \det(\nabla^2\phi_{s}) \geq s_0^n \det(\nabla^2\phi_{s})$. Set $K = \{m_{s}\leq w_{s} \leq m_{s} + 1\}$, $K_\mu = \{m_{s}\leq w_{s} \leq m_{s} + \mu\}$ and $V_\mu = Vol(K_\mu)$. From the equation, $\det(\nabla^2\phi_{s}) = e^{-w_{s} - c\cdot \nabla \varphi_s}$ and so on $K$,
\begin{align*}
\det(\nabla^2 w_{s}) &\geq s_0^n \det(\nabla^2\phi_{s})\\
& = s_0^n e^{-w_{s} - c\cdot \nabla\varphi_s}\\
& \geq Ce^{-m_{s}} ,
\end{align*}
\noindent where $C$ only depends on $s_0$ and $\sigma$ which is an upper bound for $|\nabla \varphi_s|$. So, Lemma \ref{conv anal lemma} applied to $v = w_{s} - m_{s}$ implies that $Vol(K) \leq Ce^{m_{s}/2}$. But $K_\mu \subseteq \mu K$, where by $\mu K$, we mean dilation with center $\rho_{s}$. So we have the volume estimate \\
$$V_\mu \leq C\mu^ne^{m_{s}/2}.$$
\noindent Now,
\begin{align*}
\nu^{-1} \leq Vol(X) &= \int_{\mathbb{R}^n}{\det(\nabla^2\phi_{s})\, d\rho} \\
&=\int_{\mathbb{R}^n}{e^{-w_{s}- c\cdot\nabla\varphi_s} \,d\rho}\\
&\leq Ce^{-m_{s}}\int_{0}^{\infty}{e^{-\mu}V_\mu \,d\mu}\\
&\leq Ce^{-m_{s}/2}
\end{align*}
\noindent and so $m_{s} \leq C(n,s_0,\nu, \sigma)$. For the lower bound, notice that $\nabla w_{s}(\mathbb{R}^n) = P - \tau $ and so $|\nabla w_{s}| \leq 2\sigma$. This implies that K contains a ball of radius $1/2\sigma$. But the volume of $K$ is bounded above by $Ce^{m_{s}/2}$ and so we immediately have a lower bound for $m_{s}$.  Hence (a) is proved with $C=C(n,s_0,\nu,\sigma)$.

Suppose now there exists a point $\rho \in K$ such that $|\rho-\rho_{s}| = R$. Because $B = B(\rho_{s},1/(2\sigma)) \subseteq K$, by convexity, the entire cone $\mathcal{\kappa}$ with vertex at $\rho$ and base as $B$ lies inside $K$. So, $Vol(K) \geq CR$, where $C$ depends only on dimension and $\sigma$. But $Vol(K) \leq Ce^{m_{s}/2}$ and so is less than some fixed constant $C$ by part (a). Hence $R$ is uniformly bounded. That is, there exists a uniform $R$ such that $K \subseteq B(\rho_{s},R)$. But then, convexity implies that $K_\mu \subseteq B(\rho_{s},\mu R)$. From this and the lower bound on $m_{s}$, it easily follows that
$$w_{s} \geq \frac{1}{R} |\rho-\rho_{s}| - C. $$

\medskip

\noindent This proves (b) with $\zeta = 1/R$.\\

\noindent \textbf{Step 2.} We now claim that, there exists uniform constant $C$ such that
\begin{equation}
|\rho_{s}| \leq C.
\label{min pt}
\end{equation}
\noindent We first observe,
\begin{align*}
0 &= \int_{\mathbb{R}^n}{\nabla(e^{-w_{s}}}) = -\int_{\mathbb{R}^n}{[s(\nabla \varphi_{s} -\tau) + (\alpha - s)(\nabla\hvarphi - \tau)]e^{-w_{s}} \,d\rho}\\
   % &=- \int_{P}{(x - \tau)e^{c\cdot x}\,dx}   -\int_{\mathbb{R}^n}(\alpha - s)(\nabla\hvarphi - \tau)]e^{-w_{s}}\\
    &= -\int_{\mathbb{R}^n}(\alpha - s)(\nabla\hvarphi - \tau)]e^{-w_{s}},
\end{align*}
\noindent where we use the change of coordinates $x=\nabla \varphi_{s}(\rho)$  along with the equation $e^{-w_{s}}  = \det(\nabla^2\varphi_{s})e^{c\cdot \nabla\varphi_s}$ and the fact that $\vec{c}$ and $\tau$ are compatible to conclude that the first term is zero. This computation gives us the crucial identity
\begin{equation*}
\int_{\mathbb{R}^n}{(\nabla\hvarphi -\tau)e^{-w_{s}} \, d\rho} = 0,
\end{equation*}
\noindent or equivalently,
\begin{equation}
\frac{1}{\tilde V_s}\int_{\mathbb{R}^n}{(\nabla\hvarphi)e^{-w_{s}} \,d\rho} = \tau,
\label{identity}
\end{equation}
\noindent where $\tilde V_s$ is the weighted volume given by
\begin{equation*}
\tilde V_s = \int_{\mathbb R^n}{e^{-w_s} \, d\rho}  .
\end{equation*}
\noindent Note that when the Futaki invariant vanishes, this is precisely the identity in the paper of Wang and Zhu since in that case $\tau$ is the barycenter which is zero.

Suppose the claim is false i.e for all $M > 0$ there exists a pair $(s, \rho_s)$ with $|\rho_{s}|>M$. Applying $l_j$ to both sides of the identity \eqref{identity},
\begin{equation}
\frac{1}{\tilde V_s}\int_{\mathbb{R}^n}{l_j(\nabla\hvarphi)e^{-w_{s}} \,d\rho} = l_j(\tau) > \delta
\label{delta}
\end{equation}
\noindent for some $j$ and some $\delta>0$. Fix an $\epsilon > 0$. From the estimates in the previous step there exists an $R_\epsilon > > 1$ such that
\begin{equation}
\int_{\mathbb{R}^n \backslash B(\rho_{s} , R_\epsilon)}{e^{-w_{s}} \,d\rho} \leq \epsilon.
\label{est. 1}
\end{equation}
Recall that as $\rho$ goes to infinity, the image under $\nabla \hvarphi$ goes to the boundary of $P$. So, by hypothesis, on can choose a big $M > > 1$ such that $|\rho_{s}| > M$ and $\log{(l_j(\nabla \hvarphi (\rho_{s})))} < -M$ for some $s$ and some face $l_j$.  By the gradient estimate in Lemma \ref{ggrad} there exists a constant $\Lambda$ (which does not depend on $s$) such that on $B = B(\rho_{s}, R_\epsilon)$
\begin{equation}
\log{(l_j(\nabla \hvarphi (\rho)))} < -M + \Lambda R_\epsilon < \frac{-M}{2} < \log \epsilon
\label{est. 2}
\end{equation}
\noindent for $M$ sufficiently big. So combining \ref{est. 1} and \ref{est. 2} we estimate the integral in \ref{delta},
\begin{align*}
\frac{1}{\tilde V}\int_{\mathbb{R}^n}{l_j(\nabla\hvarphi)e^{-w_{s}}} &= \frac{1}{\tilde V}\int_{B}{l_j(\nabla\hvarphi)e^{-w_{s}}} + \frac{1}{\tilde V}\int_{\mathbb{R}^n \backslash B}{l_j(\nabla\hvarphi)e^{-w_{s}}} \\
\vspace{0.7in}
&\leq \epsilon + C\epsilon,
\end{align*}
\noindent where $C$ only depends on an upper bound for the image of $P$ under $l_j$ and a lower bound for the total volume of $X$. Now choose $\epsilon$ small enough so that $\epsilon + C\epsilon < \delta/2$. But then, this contradicts (\ref{delta}), completing Step 2.\\

\noindent \textbf{Step 3.} We first observe the elementary identity from convex analysis
\begin{equation}
\sup_{\mathbb{R}^n}{|\varphi_{s} - \hvarphi|} = \sup_{P}{|u_{s} - \hat u|}.
\end{equation}
\noindent So, to complete the proof, one only needs to control the $C^0$ norm of $u_{s}$ since from the definition it is easy to see that the bound for $\hat u$ only depends on $\be_j$ and an upper bound on $l_j$.
From (\ref{key estimate}) and (\ref{min pt}),
\begin{equation}
w_{s}(\rho) \geq \zeta |\rho| - C .
\label{w estimate}
\end{equation}

\noindent Let $u_{s}$ be the Legendre transform of $\varphi_{s}$, then for any $p>n$,
\begin{align*}
\int_{P}{|\nabla u_{s}|^p\,dx} &= \int_{\mathbb{R}^n}{|\rho|^p e^{-w_{s} -c\cdot\nabla \varphi_s}\,d\rho}\\
&\leq C\int_{\mathbb{R}^n}{|\rho|^p e^{-\zeta |\rho|}\,d\rho}\\
&\leq C(p).
\end{align*}

\noindent By Morrey's inequality  $osc_{\bar P}u_{s} < C$ for some $C$ independent of $s$. Now, if we set $x_{s} = \nabla\varphi_{s}(\rho_{s})$,  then,
$$u_{s}(x_{s}) = \rho_{s} \cdot x_{s} - \varphi_{s}(\rho_{s}).$$

\medskip

\noindent The first term is clearly bounded from Step-2. Moreover by Step-1, $w_{s}(\rho_{s})$ is bounded. Since $\rho_s$ stays bounded, there is a uniform bound on $\hvarphi(\rho_s)$, which in turn gives a uniform bound on $\varphi_s(\rho_s)$.  This shows that $|u_{s}(x_{s})|$ is uniformly bounded. Hence the oscillation bound implies
$$|u_{s}|_{C^0(P)} \leq C.$$

\medskip

\noindent This completes the proof of the proposition.

\qed

 \noindent We are now in a position to prove Theorems \ref{theorem-1} and \ref{theorem-2}.

\bigskip

\noindent  \textit{\textbf{Proof of Theorem \ref{theorem-1}.} }

\noindent \textit{\textbf{Step 1.}} We first characterize the invariant $\mathcal{S}(X,L)$ in terms of the polytope data as follows - The polytope for the linear system $|-K_X - \al L|$ can be taken to be $P^{\al} = \{x \in \mathbb{R}^n ~|~ l^{\al}_{j} = v_j \cdot x + 1 - \al\lambda_j\}$. For any $\al > 0$ and any $j$,
\begin{align*}
 \tau \in P \hspace{0.05in} \text{with} \hspace{0.05in} 1 - \al l_j(\tau) \geq 0 \\
 \Leftrightarrow 0 \leq 1 - \al l_j(\tau) \leq 1  \\
 \Leftrightarrow 0 \leq v_j\cdot (-\al \tau) + 1 -\al\lambda_j \leq 1\\
 \Leftrightarrow 0 \leq l^{\al}_j(-\al\tau) \leq 1 .
 \end{align*}

But then divisor $D = \sum{l^{\al}_j(-\al \tau) D_j}$ is an effective divisor in $|-K_X - \al L|$ with coefficients less than 1.

 \medskip

\noindent \textit{\textbf{Step 2.}}  We next outline a proof of the existence of solutions to the soliton equation. Let $S = \{s \in [0,\al] | \exists$ a solution $\psi \in C^{3,\gamma}_{\be(\al)}$ to eqn. \ref{ma s} $\}$. By proposition \ref{calabi-yau}, $0 \in S$ and hence $S$ is nonempty. We now need to show that $S$ is both open and closed.\\

\noindent \textit{Openness}- The linearized operator for equation \ref{ma s} is $L_s = \Delta_s + \xi + sI$. Since $[D] \geq 0$, $Ric(\omega_s) > s\omega_s + \mathcal{L}_{\xi}\omega_s$. By lemma \ref{linear theory} all eigenvalues of $-L_s$ are strictly positive and hence the Fredholm alternative implies that $L_s$ is invertible. Implicit function theorem then implies that $S$ is open.\\

\noindent{\textit{$C^0$ estimates}}- Since there is a solution at $s = 0$ by openness there exists an $s_0$ such that there is a solution on $[0,s_0]$. With this choice of $s_0$, by proposition \ref{c^0 estimate} there exists a constant $C$ independent of $s$ such that

\begin{equation*}
 |\psi_s| = |\varphi_s - \hvarphi| \leq C .
\end{equation*}

\noindent\textit{$C^2$ and higher order estimates} - Once the uniform bound is obtained, the argument for the second and higher order estimates is the same as that in the proof of proposition \ref{calabi-yau}. Hence the upshot is that $S$ is nonempty, open and closed. Hence $S = [0,\al]$ and in particular $\al \in S$. This completes the proof of the second part of the theorem.

\medskip

\noindent \textit{\textbf{Step 3.}}  Finally, to complete the proof of Theorem \ref{theorem-1}, we now prove that  $\mathcal{R}_{BE} (X,L)= \mathcal{S}(X, L) $.
From the existence part of the theorem, it is easy to see that $\mathcal{R}_{BE}(X,L) \geq \mathcal{S}(X,L)$. In order to prove the reverse inequality, let $\alpha \in (0, \mathcal{R}_{BE})$. Then by definition, there exist smooth toric $\be$-conical metrics $\omega= \ddbar \varphi$ and $\eta= \ddbar \psi $, and a holomorphic vector field $\xi$ vector $\tau\in \mathbb{R}^n$, such that
\begin{equation*}
Ric(\omega) = \al\omega + \mathcal{L}_{\xi}\omega + \eta + [D]
\end{equation*}

\noindent for some smooth conical K\"ahler metric $\eta$ and some effective divisor $D$. Note that the volume form can be expressed as
\begin{equation*}
\omega^n = \frac{\Omega}{\prod_{j=1}^{N}{|s_j|^{2(1-\be_j)}_{h_j}}}
\end{equation*}
\noindent for some global volume form $\Omega$ with $\log{\Omega}$ bounded.  From this, it is clear that the divisor is given by
\begin{equation*}
D = \sum_{j=1}^{N}{(1-\be_j)D_j}
\end{equation*}
and consequently one can take the polytope for $\eta$ to be
\begin{equation*}
P^{\eta} = \{x \in \mathbb{R}^n | l_j^{\eta} = v_j\cdot x + \be_j -\al\lambda_j > 0 \hspace{0.05in} j = 1,\dotsc,N\} .
\end{equation*}

Locally on $(\mathbb{C}^*)^n$, $\omega = \ddbar\varphi$ and $\eta = \ddbar\psi$, and the corresponding real Monge-Ampere equation reads
\begin{equation*}
\det \nabla^2 \varphi = e^{-\alpha \varphi - \psi - c \cdot  \nabla \varphi - \tau\cdot \rho}
\end{equation*}
for some $\tau \in \mathbb{R}^n$. As before, we take $\varphi $ so that $\nabla \varphi (\mathbb{R}^n) = P$. Furthermore we normalize $\psi$ so that  $\nabla\psi(\mathbb{R}^n) = P^{\eta}$. With this normalization, we claim that $\tau = 0$.

Since $\nabla \varphi$ is bounded, It suffices to prove that
\begin{equation}
|\log{\det \nabla^2\varphi} + \al\varphi + \psi|
\label{tau}
\end{equation}
\noindent is bounded. Let $\overline \varphi = \al\varphi + \psi$ be the potential for the smooth $\be$-conical metric $\overline\omega = \al\omega + \eta$. Then $\overline\omega^n/\omega^n$ is a global bounded function. This is because both the metrics have the same poles at the divisors. Consequently it is enough to show that
\begin{equation*}
|\log{\det \nabla^2\overline\varphi} + \overline\varphi|
\end{equation*}
\noindent is bounded. But, as in the proof of Lemma \ref{cone angles},
\begin{align*}
|\log{det \nabla^2\overline\varphi} + \overline\varphi| &\leq |\sum_{j=1}^{N}{\Big (1 + \frac{x\cdot v_j}{\be_j}\Big)\log {\overline l_j} } | + C\\
&\leq  |\sum_{j=1}^{N}{\Big (1 - \frac{\overline l_j(0)}{\be_j}\Big)\log {\overline l_j} } | + C \\
&\leq C,
\end{align*}
\noindent where $\overline l_j(x) = v_j\cdot x + \be_j$ and we used the fact that $\overline l_j \log{\overline l_j}$ is a bounded function in the second line. Note that the polytope for $\overline \varphi$ is given precisely by the intersection of $\overline l_j > 0$ for $j=1,\dotsc , N$. This completes the proof of \eqref{tau} and hence proves the claim that $\tau = 0$. But then using the integration by parts trick from the proof of Lemma \ref{cone angles}
\begin{align*}
0 = \int_{\mathbb{R}^n} \nabla ( e^{-\alpha \varphi - \psi} ) d \rho = - \alpha \int_P x e^{c\cdot x} dx - \int_P \nabla \psi e^{c\cdot x} dx.
\end{align*}
\noindent So, if we set
$$\bar \tau = \frac{\int_P x e^{c\cdot x} dx}{\int_P e^{c\cdot x} dx} = \frac{ - \int_P \nabla \psi e^{c\cdot x} dx}{\alpha \int_P e^{c\cdot x} dx} .$$
\noindent Obviously, $\bar \tau \in P$ and applying $l_j$ , we have
$$1- \alpha  l_j( \bar\tau) = \frac{\int_P( 1+  v_j\cdot \nabla \psi  - \alpha \lambda_j) e^{c\cdot x}dx }{ \int_P e^{c\cdot x} dx} \geq 1- \beta_j \geq 0.$$
\noindent where we used the definition of $P^{\eta}$ for the first inequality and $D \geq 0$ for the second inequality. Hence $\al < \mathcal{S}(X.L)$ and this completes the proof of the theorem. \qed

\bigskip

\noindent \textit{\textbf{ Proof of Theorem \ref{theorem-2}.}}  This theorem follows directly from Theorem \ref{theorem-1} by taking $\tau = P_C$. For uniqueness we refer to results of Berndtsson \cite{B}. The only slightly subtle point is the existence of conical K\"ahler-Einstein metrics for $\al = \mathcal{R}(X.L)$. This follows from the fact that barycenter always stays in the interior of the polytope and hence Proposition \ref{c^0 estimate} also holds for this choice of $\al$
 (Contrast this, for instance, with the case when $\al = \mathcal{S}(X,L)$ as discussed in Example \ref{closedness at S} below).  All the higher order estimates then follow from the $C^0$ bound exactly as in the proof above.

\section{Connectedness of the space of toric conical K\"ahler-Einstein metrics}

We will prove Theorem \ref{theorem-3} in this section.

\subsection{Reducing the proof of Theorem \ref{theorem-3} to the case of one blow-up} Let us fix a toric manifold $Y$ with an ample toric line bundle $\mathcal{L}$ and corresponding polytope $P$. Let  $X$ be the blow-up  of $Y$ along a $k$-dimensional smooth toric variety $V$ with $\pi:X\rightarrow Y$ as the blow-down map. Set $L_t = \pi^{*}\mathcal{L} + tA$ for some ample toric line bundle $A$ on $X$ and for $t \in [0,1]$.  Recalling the definition of the invariant $\mathcal{R}$, we make the following simple observation

\begin{lemma}

 Let $(X_t,L_t)$ be a family of toric manifolds with ample $\mathbb{R}$-line bundles $L_t$ for $t\in (0, 1]$. Then as long as the corresponding polytopes $P_t$ stay bounded, we have
\begin{equation} \label{properties of R}
\inf_{t}{\mathcal{R}(X_t,L_t)} > 0.
\end{equation}
\end{lemma}

\begin{proof}
By Theorem \ref{theorem-1}
\begin{align*}
\mathcal{R}(X,L) &= \sup{\{\al | 1- \al l_j(P_C) > 0,~ j = 1,\dotsc,N\}}\\
&= \inf_{j}{\{\frac{1}{l_j(P_C)}\}},
\end{align*}
\noindent which stays bounded away from zero if the polytopes stay bounded.
\end{proof}

 For $t\in[0,t_0]$ we can now choose a continuous path $\al_t$ such that
\begin{equation*}
0 < \al_t < \min{(R(X,L_t), R(Y,\mathcal{L}))}
\end{equation*}
and let $\omega_t$ be the unique toric conical K\"ahler-Einstein metrics in $c_1(L_t)\cap\mathcal{K}_c(X)$ with Einstein constant $\alpha_t$. We also let $\omega_Y\in c_1(\mathcal{L})\cap \mathcal{K}_c(Y)$ be the the toric conical K\"ahler-Einstein metric on $Y$ with Einstein constant $\alpha_0$. Denoting the corresponding Riemannian metrics by $g_t $ and $g_Y$, we have the following proposition.

\begin{proposition}$(X,g_t)$ is a continuous path in the Gromov-Hausdorff topology and
$$(X,g_t) \xrightarrow{t\rightarrow 0} (Y,g_Y).$$
\label{one blow-up}
\end{proposition}

By restricting $\al_t$ to be less than $R(X,L_t)$ we ensure that $g_t$ are geodesically convex, thus facilitating the application of tools from comparison geometry. In particular, we will make use of lemma \ref{gromov's lemma}.
Taking the above proposition for granted, we now prove Theorem \ref{theorem-3}\\

\noindent \textit{\textbf{Proof of Theorem \ref{theorem-3}.}}  We fix $(X_j,\omega_j)$ for $j = 0,1$ as in the statement of the theorem and we let $L_j$ be the K\"ahler class of $\omega_j$ and $\alpha_j < R(X,L_j)$, and we call $(L_j,\omega_j,\al_j)$ compatible triples on $X_j$. By the factorization theorem \ref{factorization theorem}, there exist a sequence $0=t_0 < t_1 <\dotsc < t_k=1$ and pairs $(X_{t_i},f_{t_i})$ such that
$$ X=X_0 \stackrel{f_{t_1}}{\dashrightarrow }X_{t_1}\stackrel{f_{t_2}}{\dashrightarrow }\cdots\stackrel{f_{t_i}}{\dashrightarrow }X_{t_i} \stackrel{f_{t_{i+1}}}{\dashrightarrow }\cdots \stackrel{f_{t_k}}{\dashrightarrow }X_{t_k}=X_1 . $$
\noindent We start from the left and construct the family of metrics inductively. Suppose we are at stage $t_i$ i.e we have already constructed $(X_{t_i}, L_{t_i},\al_{t_i},\omega_{t_i})$. Then there are two cases \\

\noindent \textit{\textbf{Case-1}} - $f_{t_{i+1}}$ is a blow-down map. \\

\noindent   On $X_{t_{i+1}}$, we  take an arbitrary choice of compatible triples $(L_{t_{i+1}}, \al_{t_{i+1}}, \omega_{t_{i+1}})$. Then we connect this to $(X_{t_i}, L_{t_i},\al_{t_i},\omega_{t_i})$ in two steps. Fix a $ \mu \in (t_i,  t_{i+1})$ and ample line bundle $A$ on $X_{t_i}$.

\noindent{\textit{Step-1}} For $t\in [\mu,t_{i+1}]$, set
\begin{align*}
X_t &= X_{t_i}\\
 L_t &= f_{t_{i+1}}^*L_{t_{i+1}} + (t_{i+1}-t)A\\
\al_t&< \min({R(X_t,L_t),R(X_{t_{i+1}},L_{t_{i+1}})})
\end{align*}
\noindent where we choose $\al_t$ to be continuous. We now let $\omega_t$ be the $\al_t$ - conical K\"ahler-Eisntein metric in $L_t$. Then by Proposition \ref{one blow-up}, $(X_t,\omega_t)$ is continuous for $t\in[\mu,t_{i+1})$ and converges to $(X_{t_{i+1}},\omega_{t_{i+1}})$ in the Gromov-Hausdorff topology.

\noindent \textit{Step-2} - For $t\in[t_i,\mu]$ set
\begin{align*}
X_t &= X_{t_i}\\
L_t &= \frac{\mu - t}{\mu - t_i}L_{t_i} + \frac{t-t_i}{\mu-t_i}L_\mu\\
\al_t &< R(X_t,L_t)
\end{align*}
\noindent Again, let $\omega_t$ be the corresponding $\al_t$ - conical K\"ahler-Einstein metrics. Since in this case, $L_t$ is uniformly K\"ahler all the estimates of Proposition \ref{one blow-up} go through and we in fact get that $(X_{t_i},\omega_t)$ is continuous in $t$ in the smooth topology on $X_{t_i}$.

%\noindent \textit{\textbf{Case 1.  $f_{t_{i+1}}$ is a blow-down map. }} On $X_{t_{i+1}}$, we  take an arbitrary choice of compatible triples $(L_{t_{i+1}}, \al_{t_{i+1}}, \omega_{t_{i+1}})$. Then we connect this to $(X_{t_i}, L_{t_i},\al_{t_i},\omega_{t_i})$ in two steps. Fix a $ \mu \in (t_i,  t_{i+1})$ and ample line bundle $A$ on $X_{t_i}$. For $t\in [\mu,t_{i+1}]$, set
%
%\begin{align*}
%X_t &= X_{t_i}\\
 %L_t &= f_{t_{i+1}}^*L_{t_{i+1}} + (t_{i+1}-t)A\\
%\al_t&< \min({R(X_t,L_t),R(X_{t_{i+1}},L_{t_{i+1}})}),
%\end{align*}
%
%\noindent where we choose $\al_t$ to be continuous. We now let $\omega_t$ be the toric conical K\"ahler-Eisntein metric in $L_t$ with Einstein constant $\alpha_t$. Then by Proposition \ref{one blow-up}, $(X_t,\omega_t)$ is continuous for $t\in[\mu,t_{i+1})$ and converges to $(X_{t_{i+1}},\omega_{t_{i+1}})$ in the Gromov-Hausdorff topology.

\medskip
\noindent{\textit{\textbf{Case 2.  $f_{t_{i+1}}$ is a blow-up map. }}}   This can be treated by the same argument as in Case 1 by moving $t$ backward from $t_{i+2}$ to $t_{i+1}$.

\medskip

The smoothness of $g_t$ on the complex torus $(\mathbb{C^*})^n$ follows if we take $\alpha_t$ to be a smooth path in $t$.

\qed

\subsection{Uniform estimates and proof of proposition \ref{one blow-up}} In this section we prove Proposition \ref{one blow-up}, thereby completing the proof of Theorem \ref{theorem-3}. Throughout the section we fix an $\al>0$, such that $\alpha \in (0, \min (R(X, L_t), R(Y, \mathcal{L})))$.

Without loss of generality, we can assume that  the polytope $P$ that induces the toric manifold $Y$ is given by $(N-1)$ defining functions $l_j(x) = v_j\cdot x + \lambda_j \geq 0$, $j = 1, \dotsc, N-1$.  Let $P^{A}$ be the poytope corresponding to the ample line bundle $A$ on $X$ with $N$ defining functions $l^A_j(x) = v_j\cdot x + \lambda^A_j$ for $j=1,\dotsc,N$. The blow-up process corresponds to the $(n-1)-$dimensional face given by $l_N$  contracting to a $k$-dimensional face given by the intersection of $(n-k)$ co-dimension one faces, say $l_1,\dotsc ,l_{n-k}$. We denote the divisor corresponding to $l_j$ on $X$ by $D_j$ with defining section $s_j$, while on $Y$ we denote the divisor corresponding to $l_j$ by $\tilde D_j$ and the corresponding section by $\tilde s_j$. Then it follows from the definition of blow-ups that
\begin{equation}
\begin{cases}
 \pi^*\tilde D_j = D_j + D_N, \hspace {0.2in} j=1,\dotsc ,n-k, \\
 \pi^*\tilde D_j = D_j, \hspace {0.65 in}  j=n-k+1,\dotsc,N-1,
 \label{total transform}
\end{cases}
\end{equation}
\noindent where as before $D_j$ denotes the divisor corresponding to $l_j$ and $\pi^*$ is the total transform. Using this fact, one can explicitly write down the polytope $P^t$ for $L_t$ by defining
\begin{equation*}
\begin{cases}
l^t_j(x) = v_j\cdot x + \lambda_j + t\lambda^A_j, \hspace {1.25in} j=1,\dotsc ,N-1,  \\
l^t_N(x) = v_N\cdot x + (\sum_{j=1}^{n-k}{\lambda_j}) + t\lambda^A_N,
\end{cases}
\end{equation*}
where $v_N=\sum_{j=1}^{n-k} v_j.$

We denote the barycenters of the evolving polytopes by $P_C^t$ and the corresponding angles by $\be_j^t = \al l_j^t(P_C^t)$. We then set $l_j^0, P_C^0$ and $\be_j^0$ to be the limit of the respective quantities as $t$ goes to zero. We first prove an important identity that will be very useful, among other things, in proving that the limiting Monge-Ampere equation descends to the Einstein equation on Y.

\begin{lemma}
\begin{equation}
 (1-\be_N^0) -\sum_{j=1}^{n-k}{(1-\be_j^0)} =-( n-k-1) .
\label{eq for beta}
\end{equation}
\label{lemma for beta}
\end{lemma}

\begin{proof}
Since $D_N$ is obtained by blowing up the intersection of $D_1,\dotsc, D_{n-k}$, it is well known that
\begin{equation*}
v_N = \sum_{j=1}^{n-k}{v_j} .
\end{equation*}
\noindent Now at $t=0$, there are $(n-k+1)$ affine linear functions $l_1,\dotsc ,l_{n-k}$ and $l_N$ vanishing on a $k$- dimensional face (see the figure given below). So they must be linearly related i.e there exist real numbers $a_j$ so that
\begin{equation*}
l_N^0 = \sum_{j=1}^{n-k}{a_jl_j^0}.
\end{equation*}
\noindent But then, since $v_j$'s are linearly independent, the two equations together force the $a_j$'s to be one i.e
\begin{equation*}
l_N^0 = \sum_{j=1}^{n-k}{l_j^0} .
\end{equation*}
\noindent The lemma now follows.

\end{proof}

\noindent{\bf Example.}
Let $X = \mathbb{P}^2\#\overline{\mathbb{P}^2}$ and $Y=\mathbb{P}^2$. On $Y$ we take $\mathcal{L}$ to be the anti-canonical bundle and the corresponding $P\subset \mathbb{R}^2$ to be defined by $\{x+1 \geq 0, y+1\geq 0, 1-x-y \geq 0\}$. One can view $X$ as a projective bundle over $\mathbb{P}^1$ with a zero section $D_0$ and a section  $D_{\infty}$ at infinity. We take $A$ to be $2[D_{\infty}] - [D_0]$, with the polytope $P^A$ given by $\{x\geq 0, y \geq 0, 2 - x- y \geq 0, -1+ x + y \geq 0\}$. It follows from the Nakai criteria that $A$ is ample.  Then the polytope for $L_t$ is given by the inequalities $\{x+1\geq 0, y+1\geq 0, 1+2t - x- y \geq 0, 2-t + x+y \geq 0\}$. Computing the $\be_j^0$ for this example we see that  $\be_1^0 = 1, \be_2^0 = 1, \be_3^0 = 1, \be_4^0 = 2$. One can now easily verify Lemma \ref{lemma for beta} for this simple example.

\bigskip

Now let $\omega_t$ and $\omega_Y$ be the unique toric conical K\"ahler-Einstein metrics with Einstein constant $\al$ on $X$ and $Y$ in the class $c_1(L_t)$ and $c_1(\mathcal{L})$ respectively. In section 3, we worked with conical reference metrics coming from the symplectic potential. However, for dealing with convergence issues as the K\"ahler class degenerates, it is more convenient to work with smooth reference forms. So we pick a K\"ahler form $\tomega_Y \in c_1(\mathcal{L})$ and a K\"ahler form $\chi \in c_1(A)$. More explicitly, by taking the embedding of $Y$ into a big projective space via the sections coming from the lattice points of $P$,  we can set $\tomega_Y$ to be the pull-back of the Fubini-Study metric. One can make a similar choice for $\chi$.  We then set $\tomega_t = \pi^{*}\tomega_Y + t\chi$. Clearly, there exist locally bounded functions $\psi_t$ such that $\omega_t = \tomega_t + \sqrt{-1}\partial\bar\partial \psi_t$. We similarly have a potential $\psi_Y$ for $\omega_Y$.

\begin{lemma} \label{oscpsi}
There exists a uniform constant $C$ independent of $t$ such that
\begin{equation}\label{uniform oscillation bound}
\sup_X{\psi_t} - \inf_X{\psi_t}  \leq C.
\end{equation}
\end{lemma}

\begin{proof}
We fix a volume form on $X$ say $\Omega = \chi^n$. Recall from section 3, that $\omega_t$ is given locally on $(\mathbb{C}^*)^n$ by $\omega_t = \ddbar\varphi_t$, where $\varphi_t$ is a function only of $\rho \in \mathbb{R}^n$ and satisfies the real Monge-Ampere
\begin{equation*}
\det(\nabla^2\varphi_t) = e^{-\al(\varphi_t - P_C^t\cdot \rho)} = e^{-w_t}
\end{equation*}
\noindent The volume for $\omega_t^n$ is given by
\begin{equation*}
\omega_t^n = (\det{\nabla^2\varphi_t}) d\rho_1\wedge \dotsc \wedge \rho_n \wedge d\theta_1\dotsc\wedge d\theta_n
\end{equation*}
All the estimates in the proof of Proposition \ref{c^0 estimate} remain uniform under small perturbations of the polytope. In particular, the estimate (\ref{w estimate}) holds with constants $\zeta$ and $C$ independent of $t$. That is
\begin{equation*}
\det(\nabla^2\varphi_t) < Ce^{-\zeta |\rho|}.
 \end{equation*}
Similarly on $(\mathbb{C}^*)^n$, $\chi = \ddbar\phi$. Since $\chi$ is the pull back of the Fubini-Study metric, one can take
\begin{equation*}
\phi = \log{(\sum_{\nu \in P^A \cap \mathbb{Z}^n}{e^{\nu\cdot \rho}})} .
\end{equation*}
By an elementary calculation,  there exist constant $B_1, B_2, B_3, B_4$ depending only on $P^A$ such that
$$  B_3 e^{-B_1 |\rho|} < \det(\nabla^2 \phi) < B_4  e^{ - B_2 |\rho|}    .$$
Now, consider the trivial identity
\begin{equation}
(\tomega_t + \ddbar\psi_t)^n = \omega_t^n = \frac{\omega_t^n}{\Omega}\Omega .
\label{this lemma}
\end{equation}
\noindent We claim that $\omega_t^n/\Omega$ is in $L^{1+\epsilon}(X,\Omega)$ for some $\epsilon>0$. This is because
\begin{align*}
\int_{X}{\Bigg(\frac{\omega_t^n}{\Omega}\Bigg)^{1+\epsilon}\Omega} &= \int_{(S^1)^n}\int_{\mathbb{R}^n}{\Bigg(\frac{\det(\nabla^2\varphi_t)}{\det(\nabla^2\phi)}\Bigg)^{1+\epsilon} \det(\nabla^2\phi)\, d\rho d\theta}\\
&\leq C\int_{\mathbb{R}^n}{\Bigg(\frac{e^{-\zeta|\rho|}}{e^{-B_1|\rho|}}\Bigg)^{\epsilon}e^{-B_2|\rho|}\, d\rho}\\
&\leq C\int_{\mathbb{R}^n}{e^{ -(B_2 +  \epsilon(B_1-\zeta)) |\rho|}\, d\rho}\\
&\leq C
\end{align*}
\noindent if $\epsilon$ is small enough. Then in lieu of (\ref{this lemma}), since we have a uniform $L^{1+\epsilon}(X,\Omega)$ bound on $\omega_t^n/\Omega$, applying \cite{Ko, EGZ, ZhZ} we directly obtain a uniform bound on the oscillation of $\psi_t$.

\end{proof}

We now spend some time in deriving a complex Monge-Ampere equation satisfied by $\psi_t$. Let  $\Omega$ and $\Omega_Y$ be two fixed volume forms on $X$ and $Y$ respectively and let $\xi_Y$, $\xi_A$ be metrics on $\mathcal{L}$ and $A$ such that $\omega_Y = -\ddbar \log{\xi_Y}$ and $\chi = -\ddbar\log{\xi_A}$. One can also view  $\Omega$ and $\Omega_Y$ as metrics on $-K_X$ and $-K_Y$. We recall the adjunction formula
\begin{equation*}
K_X = \pi^*K_Y + (n-k-1)[D_N].
\label{adjunction formula}
\end{equation*}
\noindent By the $\partial\bar\partial$-lemma there exists a metric $h_N$ on $[D_N]$ such that
\begin{equation}
\Omega = \frac{\pi^*\Omega_Y}{|s_N|_{h_N}^{2(n-k-1)}}.
\label{adjunction for volume forms}
\end{equation}
\noindent Next, since $-K_Y = \al \mathcal{L} + \tilde D$, one can choose smooth hermitian metrics $\tilde h_1,\dotsc,\tilde h_{N-1}$ on $\tilde D_1,\dotsc,\tilde D_{N-1}$ such that
\begin{equation}
 \prod_{j=1}^{N-1}{\tilde h_j ^{(1-\be_j^Y)}} =\pi^*\left(  \frac{\Omega_Y}{(\xi_Y)^{\al}} \right).
 \label{OmegaY and h}
\end{equation}
\noindent Using \ref{total transform}, we then define smooth metrics on $D_j$ for $j < N$ by setting
\begin{equation}
\begin{cases}
 h_j = \pi^*\tilde h_j/  h_N & j=1,\dotsc ,n-k \\
 h_j =  \pi^*\tilde h_j &  j=n-k+1,\dotsc,N-1 .
 \label{eq for h}
\end{cases}
\end{equation}
\noindent Finally we define a family of metrics on $[D_N]$ by
\begin{equation}
h_N(t) = \Bigg(\frac{\Omega\xi_A^{-t\al}\pi^*(\xi_Y^{-\al})}{\prod_{j=1}^{N-1}{h_j^{(1-\be_j^t)}}}\Bigg)^{\frac{1}{1-\be_N^t}}.
\end{equation}
\noindent We claim

\begin{lemma} At all points of $X$,
\begin{equation}\label{conv of h}
\lim_{t\rightarrow 0}{\frac{h_N(t)}{h_N}} = 1.
\end{equation}
\end{lemma}

\begin{proof}
If we consider $\Omega$ as a metric on $-K_X$ and $\pi^*\Omega_Y$  as a pull back metric on $-\pi^*K_Y$, then by equation \eqref{adjunction for volume forms},  $\Omega = \pi^*\Omega_Y h_N^{-(n-k-1)}$.
\begin{align*}
\lim_{t\rightarrow 0}{\frac{h_N(t)}{h_N}} &= \Bigg(\frac{\Omega\pi^*(\xi_Y)^{-\al}}{h_N^{(1-\be_N^0)}\prod_{j=1}^{N-1}{h_j^{(1-\be_j^0)}}}\Bigg)^{\frac{1}{1-\be_N^0}}\\
&=\Bigg(\frac{\pi^*(\Omega_Y\xi_Y^{-\al})h_N^{-(n-k-1)}}{h_N^{(1-\be_N^0)}\prod_{j=1}^{N-1}{h_j^{(1-\be_j^0)}}}\Bigg)^{\frac{1}{1-\be_N^0}}\\
&= \Bigg(\frac{\pi^*(\Omega_Y\xi_Y^{-\al})}{\prod_{j=1}^{N-1}{\pi^*\tilde h_j^{(1-\be_j^Y)}}}\Bigg)^{\frac{1}{1-\be_N^0}}\\
&=1,
\end{align*}
\noindent where we used lemma \ref{eq for beta}, equation \eqref{eq for h} in line three and equation \eqref{OmegaY and h} in line four.

\end{proof}

\noindent By applying logarithm and taking $\partial\bar\partial$ we see that $h_1,\dotsc,h_{N-1}$ and $h_N(t)$ satisfy

\begin{equation*}
-\partial\bar\partial\log{\Omega} = \al\tomega_t - \sum_{j=1}^{N-1}{(1-\be_j^t)\partial\bar\partial\log{h_j}} - (1-\be_N^t)\partial\bar\partial\log{h_N(t)}.
\end{equation*}

\noindent The purpose of the above constructions is that $\psi_t$ and $\psi_Y$ now satisfy, possibly after modification by some constant, the following Monge-Ampere equations

\begin{align}
(\tomega_t + \ddbar\psi_t)^n &= \frac{e^{-\al\psi_t}\Omega}{|s_N|^{2(1-\be_N^t)}_{h_N(t)}\prod_{j=1}^{N-1}{|s_j|^{2(1-\be^t_j)}_{h_j}}} \tag*{$(*)_t$} \label{ma t},  \\
(\tomega_Y + \ddbar\psi_Y)^n &= \frac{e^{-\al\psi_Y}\Omega_Y}{ \prod_{j=1}^{N-1}{|\tilde s_j|_{\tilde h_j} ^{2(1-\be_j^Y)}}} \tag*{$(*)_Y$} \label{ma Y}.
\end{align}

\noindent We remark that since modification by a constant doesn't change the oscillation, the estimate of lemma \ref{uniform oscillation bound} holds for this modified $\psi_t$. Immediately, we have the following corollary from Lemma \ref{oscpsi} because the total volume of $(X , \omega_t)$ is $[L_t]^n$ and is uniformly bounded .

\begin{corollary} There exists a unique constant $C>0$ such that for all $t\in (0,1]$,
\begin{equation}
||\psi_t||_{L^\infty(X)} \leq C.
\end{equation}

\end{corollary}

We next prove uniform estimates away from $D$ on all derivatives of $\psi_t$

\begin{proposition}
For all $l\geq 0$ and $K \subset\subset X\backslash D$, there exist constants $C_{K,l}$ independent of $t$ such that
\begin{equation}\label{uniform local estimates}
||\psi_t||_{C^l(K)}< C_{K,l} .
\end{equation}
Here the norm is with respect to some fixed reference metric.
\end{proposition}

\begin{proof}  Since the usual arguments for $C^{2,\gamma}$ estimates, using the theory of Evans, Kryllov and Safanov, are local in nature they can be used in the present context. Hence, to prove the proposition, we only need uniform  $C^2$ estimates.

We follow the argument in \cite{ST0} using Tsuji's trick \cite{Ts}.  By Kodaira's lemma, $L_\epsilon = \pi^*\tomega_y - \epsilon [D_N] > 0$ for small $\epsilon >0$. So, we pick a new smooth hermitian metric $\xi_N$ on $[D_N]$ such that $\eta = \pi^*\tomega_Y + \epsilon\partial\bar\partial\log{\xi_N} > 0$ and consider
\begin{equation*}
Q_t = \log\Bigg({|s_N|^{2(1 + A\epsilon)}_{\xi_N}\prod_{j=1}^{N-1}{|s_j|_{h_j}^2} tr_{\eta}{\omega_t}}\Bigg) - A\psi_t
\end{equation*}
 for some big constant $A$ to be chosen later. We note that $Q$ goes to negative infinity near $D$. This is because the order of poles of $\omega_t$ near each $D_j$ is $2(1-\be_j^t)$ which is strictly less than two. So for each $t$, the maximum is attained in $X\backslash D$. Following Yau, we compute $\Delta_t Q_t$ where $\Delta_t$ is the Laplacian with respect to $\omega_t$. Since on $X\backslash D$, $Ric(\omega_t) = \al\omega_t$, standard calculations show that there exists a constant $C$ depending only on the dimension and curvature of $\eta$ such that
\begin{equation*}
\Delta_tQ_t \geq -Ctr_{\omega_t}{\eta} + (1+ A\epsilon)\Delta_t\log{\xi_N} + \sum_{j=1}^{N-1}{\Delta_t\log{h_j}}  + Atr_{\omega_t}\tomega_t -C.
\end{equation*}
\noindent Also, there exists a constant $C'$ independent of $t$ such that
\begin{align*}
\Delta_t\log{\xi_N} &> -C'tr_{\omega_t}\eta, \\
\Delta_t\log{h_j} &> -C'tr_{\omega_t}\eta.
\end{align*}
\noindent Combining this with the above estimate
\begin{align*}
\Delta_tQ_t &\geq -Ctr_{\omega_t}{\eta} + Atr_{\omega_t}(\tomega_t + \epsilon \partial\bar\partial\log{\xi_N}) - C\\
&= -Ctr_{\omega_t}{\eta} + Atr_{\omega_t}(\eta + t\chi) -C\\
&> tr_{\omega_t}{\eta} - C,
\end{align*}
\noindent where we choose $A = C+1$. So, at the maximum point of $Q_t$, $tr_{\omega_t}\eta (p_t) < C$. Now, standard arguments show
\begin{equation*}
\Bigg({|s_N|^{2(1 + A\epsilon)}_{\xi_N}\prod_{j=1}^{N-1}{|s_j|_{h_j}^2}\Bigg) tr_{\eta}{\omega_t}}  < Ce^{\sup{\psi_t} - \inf{\psi_t}}\Bigg(|s_N|^{2(1 + A\epsilon)}_{\xi_N}\prod_{j=1}^{N-1}{|s_j|_{h_j}^2}\frac{\omega_t^n}{\eta^n}\Bigg)(p_t).
\end{equation*}
\noindent Using the equation and the oscillation bound on $\psi_t$ and the fact that $\be^t_j < 1$,
\begin{equation*}
tr_{\eta}{\omega_t} < \frac{C}{\Bigg(|s_N|^{2(1 + A\epsilon)}_{\xi_N}\prod_{j=1}^{N-1}{|s_j|_{h_j}^2}\Bigg)}.
\end{equation*}
\noindent Hence, we have uniform second order estimates away from $D$ and this completes the proof of the proposition.

\end{proof}

With the above uniform local estimates and the uniqueness of $\omega_Y$, we have the following local uniform convergence away from the divisors.

\begin{proposition} For any compact subset $K\subset\subset X\setminus D$, we have the following uniform convergence
\begin{equation*}
\omega_t \xrightarrow{C^{\infty}(K)} \omega_Y.
\end{equation*}

\end{proposition}

Using Moser iteration, one can in fact show that $\psi_t$ converges to $\pi^*\psi_Y$ globally in $L^\infty$.
We now have to prove the global convergence, in Gromov-Hausdorf topology, of $(X, \omega_t)$ to $(Y, \omega_Y)$.

\bigskip

\noindent \textit{\textbf{Proof of Proposition \ref{one blow-up}.}} We let  $t\rightarrow 0$. Fix an $\epsilon > 0$.  Let $E$ be a tubular neighborhood of $D \subset Y $ such that $A = Y\backslash E$ is $\epsilon$-dense in  $Y$ with respect to the metric $g_Y$. Note, that since $X$ and $Y$ are bi-holomorphic away from $D$, $A$ can be identified as a subset of $X$.  We also pick $E$ close enough to $D$ so that $Vol_{g_Y}(E) < \epsilon^{4n}$ and we set $\tilde E = \pi^*(E)$. Finally, we denote the distances with respect to $g_t$ and $g_Y$, by $d_t$ and $d_Y$ respectively.\\

\noindent \textit{\textbf{Claim 1.}} For $t$ small enough, $A=Y\backslash E = X \backslash \tilde E$ is $\epsilon$ - dense in $(X,g_t)$.
 If not, then there exists a sequence $t_k \rightarrow 0$ and points $x_k \in \tilde E$ such that $B_{g_{t_k}}(x_k,\epsilon) \subset \tilde E$. Using volume comparison,  uniform diameter bounds and the fact that the volumes converge, for small $t_k$
\begin{equation*}
\kappa \epsilon^{2n} < Vol_{g_k}(B_{g_k}(x_k,\epsilon)) < Vol_{g_k}(\tilde E) < 2Vol_{g_Y}(E) < 2\epsilon^{4n}
\end{equation*}
\noindent for some constant $\kappa$ if $k$ is sufficiently large. But if $\epsilon$ is small enough, this is  a contradiction. \\
\noindent \textbf{\textit{Claim 2.}} There exists a $t(\epsilon)$ such that for all $0<t<t(\epsilon)$ and for all $p,q \in A$, $$d_{t}(p,q) < d_Y(p,q) + \epsilon . $$
\noindent \textit{Proof}. By the geodesic convexity of $Y\backslash D$, one can choose a small tubular neighborhood, $T \subset E$ of $D$ in $Y$ such that any two points in $A$ can be connected by a $g_Y$ - minimal geodesic in $Y \backslash T$. Set $\tilde T  = \pi^{-1}(T)$. Let $\gamma$ be a $g_Y$ - minimal geodesic connecting $p$ and $q$ lying in $Y\backslash T$. Since the metrics converge uniformly on compact sets of $X\setminus \tilde E$, for $t$ sufficiently small,
\begin{equation*}
d_t(p,q) < \mathcal{L}_t(\gamma) < \mathcal{L}_Y(\gamma) + \epsilon = d_Y(p,q) + \epsilon,
 \end{equation*}
\noindent where $\mathcal{L}$ denotes the length functional. \\

\noindent \textit{\textbf{Claim 3.}} There exists a $t(\epsilon)$ such that for all $0<t<t(\epsilon)$ and for all $p,q \in A$, $$d_{t}(p,q) > d_Y(p,q) - \epsilon .$$

\noindent \textit{Proof}.  The proof of this claim relies on the generalization of Gromov's lemma to the conical setting (Lemma \ref{grom}). We once again choose a  tubular neighborhood $T$ of $D$ contained in $E$ with smooth boundary such that
for all $q\in A$,
\begin{align*}
B_{g_Y}(q,\epsilon/2) &\subset Y\backslash T\\
Vol_{g_Y}(\partial T) &< \delta/2
\end{align*}
\noindent and set $\tilde T  = \pi^*(T)$.  $D$ being of real co-dimension two, one can choose $\delta$ to be as small as needed. Since, away from $D$, the metric converges uniformly, we can assume that $Vol_{g_t}(\partial T) < \delta$ by choosing $t$ sufficiently small. Furthermore, since $d_{g_Y}(q,\partial T) > \epsilon/2$, once again by the uniform convergence of the metric on $X\backslash \tilde T$, for small $t$, $d_{g_t}(q,\partial \tilde T) > \epsilon/4$, i.e.,  $B_{g_t}(q,\epsilon/4) \subset X\backslash \tilde T$.

We claim that there exists at least one minimal geodesic from $p$ to a point in $B_{g_t}(q,\epsilon/4)$ that lies entirely in $X\backslash \tilde T$. If not, then by Gromov's lemma there exists a constant $c$ uniform in $t$ (but  depending on $\epsilon$) such that
\begin{equation*}
\kappa \epsilon^{2n} < Vol_{g_t}(B_{g_t}(q,\epsilon/4)) < c Vol_{g_t}(\partial \tilde T) < c\delta.
\end{equation*}
\noindent Letting $\delta$ go to zero, we get a contradiction.

So there exists at least one $g_t$ - minimal geodesic $\gamma_t$ connecting $p$ to a point $\tilde q_t \in B_{g_t}(q,\epsilon/4)$. By compactness, there exists a $\tilde q \in B_{g_Y}(q,\epsilon/2)$ such that $\tilde q_t \rightarrow \tilde q$ and moreover the geodesics $\gamma_t$ converge to a curve, denoted by $\gamma$, joining $p$ to $\tilde q$.

\begin{align*}
d_{g_t}(p,q) &> \mathcal{L}_{g_t}(\gamma) - \epsilon/4\\
&>\mathcal{L}_{g_Y} (\gamma) -\epsilon/2\\
&> d_{g_Y}(p,\tilde q) - \epsilon/2\\
&> d_{g_Y}(p,q) - \epsilon
\end{align*}
\noindent and this proves Claim 3.

\medskip
\noindent Now we complete the proof of the proposition.  For sufficiently small $t$,
\begin{eqnarray*}
&&d_{GH}((X,d_t),(Y,d_Y)) \\
&\leq& d_{GH}((X,d_t), (A,d_t)) + d_{GH}((A,d_t), (A,d_Y)) + d_{GH}((A,d_Y), (Y,d_Y)) \\
&<& 3\epsilon,
\end{eqnarray*}
\noindent where we use Claim 1 to bound the first term, Claim 2 and Claim 3 to bound the second term, while the last term is bound by $\epsilon$ from the choice of $A$.  Now, letting $\epsilon$ go to zero, we see that $(X,g_t)$ converges in Gromov-Hausdorff distance to $(Y,g_Y)$.    \qed

\bigskip

\section{Discussions}

In this section, we propose some questions in relation to our main results. In Theorem \ref{theorem-1}, given the toric pair $(X, L)$, the smooth toric conical K\"ahler-Ricci soliton equation can in general be solved for $\alpha \in (0, \mathcal{S}(X,L) )$, while the smooth toric conical K\"ahler-Einstein equation can be solved for $\alpha\in (0, \mathcal{R}(X, L) ]$. The following example illustrates when the soliton equation can be solved at $\mathcal{S}(X, L)$.

\begin{example}\label{ex5}Let $X=\mathbb P^2\# \overline{\mathbb P^2}$, i.e., $\mathbb P^2$ blow-up at one point, given by a polytope $P$ defined by $l_0 = x+y+\varepsilon>0$, $l_1 = y+1>0$, $l_2 = x+1>0$, and $l_\infty = -x - y +\varepsilon>0$, where $\varepsilon>0$ is a small constant. Let $L$ be the K\"ahler class induced by $P$. The corresponding ample line bundle $L$ over $X$ is determined by the divisor $D_1 + D_2 + \varepsilon D_\infty + \varepsilon D_0$ (noting that the bundle being ample is equivalent to $\varepsilon\in (0,2)$). By Theorem 1.1, $\frac{1}{\mathcal S(X,L)}$ can be characterized as
\begin{equation}\label{eqn:aimed1}
\inf_{(x,y)\in P} \max\{x+1, y+1, x+y+\varepsilon, -x - y + \varepsilon\}.
\end{equation}
%We claim that the extremal point of the above function is $(x,y) = (-\frac{\varepsilon}{2}, -\frac{\varepsilon}{2})$ if $\varepsilon$ is small enough. To see this,

\noindent By the symmetry of $x, y\in P$, the extremal point must be at the line $y-x = 0$, hence the aimed function is reduced to
\begin{equation}\label{eqn:aimed}
\inf_{x\in (-\varepsilon/2, \varepsilon/2)} \max\{x+1, 2x+\varepsilon, -2x + \varepsilon\}.
\end{equation}
We have three cases: $\varepsilon\in (0, \frac{2}5], (\frac{2}{5}, 1), [1,2)$.\begin{enumerate}
\item When $\varepsilon\in (0,\frac{2}{5}]$, the unique extremal point of \eqref{eqn:aimed} is at $x=-\frac{\varepsilon}{2}$, or, that of \eqref{eqn:aimed1} is at $(-\frac{\varepsilon}{2}, \frac{\varepsilon}{2})$, which is at the boundary of $P$, hence the conical K\"ahler-Ricci soliton equation cannot be solved at $\mathcal S(X,L)$ for this case.
\item When $\varepsilon\in (\frac{2}{5},1)$, the unique extremal point of \eqref{eqn:aimed} is at $x=-\frac{1-\varepsilon}{3}$, and the point $(-\frac{1-\varepsilon}{3},-\frac{1-\varepsilon}{3})$ is in the interior of $P$, hence by our proof above, the conical K\"ahler-Ricci soliton equation can be solved at $\mathcal S(X,L)$ in this case.
\item When $\varepsilon\in [1,2)$, the unique extremal point of \eqref{eqn:aimed} is at $x=0$, and the origin $(0,0)$ is always in $P$, so in this case, the conical K\"ahler-Ricci soliton equation can also be solved up to $\mathcal S(X,L)$.
\end{enumerate}
\label{closedness at S}
\end{example}

It is a natural question to ask in the above example, when $\epsilon \in (0, 2/5]$, if there exists a limiting space for the K\"ahler-Ricci solitons as $\alpha$ tends to $\mathcal{S}(X, L)$. Since the Futaki invariant blows up as $\alpha $ tends to $\mathcal{S}(X,L)$, most likely the limiting space will be a complete shrinking or steady soliton with an complete end at the exceptional divisor after some appropriate scaling.

We would also like to mention that since for $\alpha \in (0, \mathcal{R}_{BE})$, the toric conical K\"ahler-Ricci soliton metric in Theorem \ref{theorem-1} is not unique and in fact, there are infinitely many of them. Let $\mathcal{KR}(X, L, \alpha)$ be the space of K\"ahler-Ricci soliton metrics $\omega\in L \cap \mathcal{K}_c(X)$ with $Ric(\omega)=\alpha \omega + \mathcal{L}_\xi \omega + [D]$, if $\alpha \in (0, \mathcal{R}_{BE}(X,L))$. Then we define
$$F_{min} (X, L, \alpha) = \inf \{Fut(X, \omega)~|~ \omega \in \mathcal{KR}(X,L, \alpha)\}, $$
where $Fut(X, \omega)$ is the Futaki invariant for the holomorphic vector field $\xi$ and $Fut(X, \omega) = \int_X |\xi|^2 \omega^n.$ After some calculations, one can show that $F_{min}(X, L, \alpha)$ is achieved for some smooth toric K\"ahler-Ricci soliton metric in $L\cap \mathcal{K}_c(X)$. We call such a conical K\"ahler-Ricci slotion the minimal K\"ahler-Ricci soliton with respect to $(X, L, \alpha)$.

Finally we propose the following regression scheme to obtain models for any log Fano variety $(X, D)$ by maximizing the Ricci curvature and Bakry-Emery-Ricci curvature.

\medskip

We first consider the Fano case.
\begin{enumerate}

\item[{\bf (1)}] Let $X$ be a Fano manifold and $D$ be a smooth simple divisor in $|-mK_X|$ for some $m \in \mathbb{Z}^+$.

\medskip

\begin{enumerate}

\item[(1.1)] Let $T= \sup\{ t~|~Ric(\omega)= t \omega + (1-t)[D/m] \text{ for a conical K\"ahler metric } \omega\}$ and $\omega_t$ be the unique K\"ahler-Einstein metric solving $Ric(\omega_t) = t \omega_t + (1-t)[D/m]$ for  $t\in (0, T)$.

\medskip

\item[(1.2)] Then $(X, D/m, \omega_t)$ converges in Gromov-Hausdorff sense to a compact metric space $(X', D', \omega')$ homeomorphic to a Fano variety $X'$ with log terminal singularities.  $D'$ is an effective $\mathbb{R}$-divisor in $-K_{X'}$ and $\omega'$ is a K\"ahler-Einstein metric solving
$$Ric(\omega')=T \omega' + (1-T)[D'],$$
in particular, $\omega'$ is smooth outside $D'$ and the singularities of $X'$.

\medskip

\item[(1.3)] We replace $(X, D/m)$ by $(X', D')$ and go back to $(1)$ and continue the procedure.

\medskip

\item[(1.4)] The regression terminates in finite steps for finitely many data $\{ (X_i, D_i), T_i\}$,  $i=0, ..., I$,  with $(X_0, D_0)= (X, D/m)$ and $T_0=0$.

\medskip

\item[(1.5)] $  (X_I, D_I, \omega_I)$ is a  $\mathbb{Q}$-Fano conical K\"ahler-Einstein space with Ricci curvature equal to $T_I =R(X, D)\leq R(X)$. We cannot increase $T_I$ by repeating the above process.  $R(X,D)=R(X)$ for a generic smooth divisor $D$ and $I=1$. If $R(X)=1$, we stop with $D_I$ being $0$, otherwise we continue with the next step.

\end{enumerate}

\medskip

\item[{\bf (2)}] Let $(X_{KE}, D_{KE}) = (X_I, D_I)$.

\medskip

\begin{enumerate}

\item[(2.1)] Let $T = \sup\{ t~|~ Ric(\omega) = t\omega + \mathcal{L}_\xi \omega + (1-t) [D] \text{ for a conical K\"ahler metric } \omega  \}$ and $(\omega_t, \xi_t) $ be the minimal K\"ahler-Ricci soliton metric solving $$Ric(\omega_t) = t\omega_t + \mathcal{L}_{\xi_t} \omega_t + (1-t) [D_t]$$ for $t\in [R(X), T)$. The effective divisor $D_t$ depends on $t$.

\medskip

\item[(2.2)] Then $(X_{KE}, D_t, \omega_t, \xi_t)$ converges in Gromov-Hausdorff sense to a compact metric space $(X', D', \omega', \xi')$ homeomorphic to a Fano variety $X'$ with log terminal singularities.  $D'$ is an effective $\mathbb{R}$-divisor in $-K_{X'}$ and $\omega'$ is a conical K\"ahler-Ricci soliton metric solving
$$Ric(\omega')=T \omega' + \mathcal{L}_{\xi'} \omega' + (1-T)[D'],$$
in particular, $\omega'$ is smooth outside $D'$ and the singularities of $X'$.

\medskip

\item[(2.3)] We replace $(X_{KE}, D_{KE})$ by $(X', D')$ and go back to $(2)$ and continue the procedure.

\medskip

\item[(2.4)] The regression terminates in finite steps for a finitely many data $\{ (X_j, D_j), T_j\}$ for $j=I, ..., J$ with $(X_I, D_I)= (X_{KE}, D_{KE} )$.

\medskip

\item[(2.5)]  $T_J= R_{BE}(X, D_I)\leq R_{BE}(X)=1$. For generic $D_0$, $D_J=0$  and $(X_{KR}, \omega_{KR})= (X_J, \omega_1)$ is a  $\mathbb{Q}$-Fano K\"ahler-Ricci soliton space solving $Ric(\omega_{KR}) = \omega_{KR} + \mathcal{L}_{\xi_{KR}}\omega_{KR}$.

\end{enumerate}

\medskip

\item[{\bf (3)}] We call $(X_{KE}, \omega_{KE}) $ when $R(X)=1$ and $(X_{KR}, \omega_{KR})$ when $R(X)<1$ the maximal model of $X$ in the sense that it maximizes  Bakry-Emery-Ricci curvature.  It should be unique and does not depend on the choice of the initial divisor $D$, while the intermediate model $(X_{KE}, D_{KE}) $ is not necessarily unique and it might depend on $D$,  if $R(X)<1$.

\end{enumerate}

\bigskip

If the above speculation holds, we also conjecture that the maximal model of a Fano manifold $X$ should coincide with the limit of the Fano K\"ahler-Ricci flow on $X$, generalizing \cite{PSSW}. The above scheme can be generalized to any log Fano variety $\{ X, D \}$ with some modifications if the polarization $-K_X-D$ satisfies log $K$-stability for the log pair $(X, -(1-\epsilon)K_X+\epsilon D)$ with sufficiently small $\epsilon>0$. This will give a unique maximal model for any log Fano variety if $\mathcal{R}_{BE}(X, -K_X-D)$  is achieved.  In this case, the divisor where the conical singularities occur changes continuously  $t$ increases so that the deforming soliton metrics are always minimizing the Futaki invariant for each $t$. The minimal model program deforms projective varieties with positive Kodaira dimension to their minimal model by birational tranformations, while in the Fano case, our regression scheme looks for a model of Fano manifold by deformations of complex structures which should be also identified as certain algebraic operations in relation to log K-stability.  A good model to test the above scheme is the example in \cite{Sze2}.

\bigskip

\noindent{\bf Acknowledgements} We would like to thank D.H. Phong, Jacob Sturm, Gabor Szekelyhidi and Daguang Chen for many stimulating discussions. The paper is also part of the theses of the first and second named authors at Rutgers University, and they would like to thank the generous support from the math department of Rutgers University.

\end{document}